\documentclass[reqno]{amsart}
\usepackage{amsmath}%
\usepackage{amsfonts}%
\usepackage{amssymb}%
\usepackage{graphicx}
\usepackage{mathrsfs}
\usepackage{hyperref}
\usepackage{fullpage}
\usepackage{color}
\usepackage{cite}
\def\0{\emptyset}

\def\n{\noindent}

\newtheorem{theorem}{Theorem}
\newtheorem{lemma}[theorem]{Lemma}

\newtheorem{claim}[theorem]{Claim}
\newtheorem{subclaim}[theorem]{Subclaim}

\newtheorem{conjecture}[theorem]{Conjecture}

\newtheorem{fact}{Fact}

\begin{document}
\title{Vertex degree sums for perfect matchings in 3-uniform hypergraphs }
\thanks{Yan Wang is supported by National Natural Science Foundation of China under Grant No. 12201400, National Key R\&D Program of China under Grant No. 2022YFA1006400 and Explore X project of Shanghai
Jiao Tong University. Yi Zhang is supported by Fundamental Research Funds for the Central Universities and Innovation Foundation of BUPT for Youth under Grant No. 2023RC49 and National Natural Science Foundation of China under Grant No. 11901048 and 12071002.}

\author{Yan Wang}
\address{School of Mathematical Sciences, CMA-Shanghai, Shanghai Jiao Tong
University, Shanghai 200240, China}
\email{yan.w@sjtu.edu.cn}

\author{Yi Zhang}
\address{ School of Sciences, Beijing University of Posts and Telecommunications, Beijing 100876, China}
\email{shouwangmm@sina.com}

\date{\today}


\begin{abstract}
Let $n \equiv 0\, (\, \text{mod } 3\,)$ and  $H_{n, n/3}^2$ be the 3-graph of order $n$, whose vertex set is partitioned into two sets $S$ and $T$ of size $\frac{1}{3}n+1$ and $\frac{2}{3}n -1$, respectively, and whose edge set consists of all triples with at least $2$ vertices in $T$. Suppose that $n$ is sufficiently large and  $H$ is a 3-uniform hypergraph of order $n$ with  no isolated vertex. Zhang and Lu [Discrete Math. 341 (2018), 748--758] conjectured that if  $\deg(u)+\deg(v) > 2(\binom{n-1}{2}-\binom{2n/3}{2})$ for any two vertices $u$ and $v$ that are contained in some edge of $H$, then $H$ contains a perfect matching or $H$ is a subgraph of $H_{n,n/3}^2$. We construct a counter-example to the conjecture. Furthermore, for all $\gamma>0$ and let $n \in 3 \mathbb{Z}$ be sufficiently large, we prove that if $\deg(u)+\deg(v) > (3/5+\gamma)n^2$ for any two vertices $u$ and $v$ that are contained in some edge of $H$,   then $H$ contains a perfect matching or $H$ is a subgraph of $H_{n,n/3}^2$. This implies a result of Zhang, Zhao and Lu [Electron. J. Combin. 25 (3), 2018].
\end{abstract}

\maketitle

\section{Introduction}

A \emph{$k$-uniform hypergraph} $H$ (in short, \emph{$k$-graph}) is a pair $(V(H),E(H))$, where $V(H)$ is a finite set
of vertices and $E(H)$ is a family of $k$-element subsets of $V$.
A \emph{matching of size $s$} in $H$ is a family of $s$ pairwise disjoint edges of $H$. If a matching covers all the vertices of $H$, then we call it a \emph{perfect matching}.
Given a set $S \subseteq V$, the \emph{degree} $\deg_{H}(S)$ of $S$ is the number of the edges of $H$ containing $S$. We simply write $\deg(S)$ when $H$ is obvious from the context. If $S=\{u\}$, then we write $\deg(u)$ instead of $\deg(S)$. Further, let $\delta_\ell(H)=\min\{\deg(S)~:~S\subseteq V(H),~|S|=\ell\}$
be the minimum $\ell$-degree of $H$.

Given integers $\ell<k\le n$ such that $k$ divides $n$, let $m_\ell(k,n)$ denote the smallest integer $m$ such that every $k$-graph $H$ on $n$ vertices with $\delta_\ell(H) \geq m$ contains a perfect matching. In recent years the problem of determining $m_\ell(k,n)$ has received much attention (see \cite{Alon}, \cite{Han,Han3,Kha1,Kha2,Kuhn1,Kuhn2,LuMi,Mar,Mitzenmacher,Ore,Pik}, \cite{Rod2,Rod1,Rod3}, \cite{Tre,TrZh13,Town2}).
More Dirac-type results on hypergraphs can be found in surveys \cite{RoRu-s, Zhao}.

A well-known result of Ore \cite{Ore} extended Dirac's theorem by determining the smallest degree sum of two non-adjacent vertices that guarantees a Hamilton cycle in graphs. Ore-type results for hypergraphs witnessed much progress several years ago. For example, Tang and Yan \cite{Tang} studied the degree sum of two $(k-1)$-sets that guarantees a tight Hamilton cycle in $k$-graphs. Zhang and Lu \cite{Yi1} studied the degree sum of two $(k-1)$-sets that guarantees a matching of size $s$ in $k$-graphs. Zhang, Zhao and Lu \cite{zhang,zhang2} determined the minimum degree sum of two adjacent vertices that guarantees a perfect matching in 3-graphs without isolated vertices (two vertices in a hypergraph are \emph{adjacent} if  there exists an edge containing both of them). Note that one may study the minimum degree sum of two arbitrary vertices and that of two non-adjacent vertices that guarantees a perfect matching -- it was mentioned in \cite{zhang} that the former equals to $2 m_1(3, n)-1$ while the latter does not exist (consider a $3$-graph whose edge set consists of all $3$-tuples containing a fixed vertex).

For $1\le \ell\le 3$, let $H^{\ell}_{n, s}$ denote the 3-graph of order $n$, whose vertex set is partitioned into two sets $S$ and $T$ of size $n- s\ell+1$ and $s\ell -1$, respectively, and whose edge set consists of all triples with at least $\ell$ vertices in $T$. Obviously all these graphs $H^{\ell}_{n, s}$  do not contain a matching of size $s$. A well-known conjecture of Erd\H{o}s ~\cite{Erd65}, verified for $3$-graphs \cite{Fra, LuMi} , implies that $H_{n,s}^1$  or $H_{n,s}^3$  is the densest $3$-graph on $n$ vertices not containing a matching of size $s$. On the other hand, K\"{u}hn, Osthus and Treglown \cite{Kuhn2} showed that for sufficiently large $n$, $H_{n,s}^1$ has the largest minimum vertex degree among all 3-graphs on $n$ vertices not containing a matching of size $s$.

\begin{theorem}\cite{Kuhn2}\label{Kuhn2}
There exists $n_0 \in \mathbb{N}$ such that if $H$ is a $3$-graph of order $n \geq n_0$ with $\delta_1(H) > \delta_1(H^1_{n, s})= \binom{n-1}{2}-\binom{n-s}{2}$ and $n \geq 3s$, then $H$ contains a matching of size $s$.
\end{theorem}

Given a 3-graph $H$, let $\sigma_2(H)$ denote the minimum $\deg(u)+\deg(v)$ among all adjacent vertices $u$ and $v$. Note that
\begin{align*}
\sigma_2(H_{n,s}^3) &= 2 \binom{3s-2}{2}, \quad
\sigma_2(H_{n,s}^1) =2\left(\binom{n-1}{2}-\binom{n-s}{2} \right) \ \text{and} \\
\sigma_2(H_{n,s}^2) & =\binom{2s-2}{2}+\left(n-2s+1\right)\binom{2s-2}{1}+\binom{2s-1}{2} = (2s-2)(n-1).
\end{align*}
Note that $\sigma_2(H_{n,s}^2) \geq \sigma_2(H_{n,s}^1)$. For those 3-graphs without isolated vertices,
Zhang, Zhao and Lu \cite{zhang2} showed $H_{n,s}^2$ to be the extremal graph without a matching of size $s$.
\begin{theorem}\cite{zhang2}
\label{the1}
There exists $n_0 \in \mathbb{N}$ such that the following holds for all integers $n\ge n_0$ and $s\le n/3$. If $H$ is a $3$-graph of order $n$ without isolated vertex and $\sigma_2(H) > \sigma_2(H_{n,s}^2)= 2(s-1)(n-1)$, then $H$ contains a matching of size $s$.
\end{theorem}
On the other hand,  $\sigma_2(H_{n,s}^2) \geq \sigma_2(H_{n,s}^3)$ if and only if $s \le (2n+4)/9$. One may allow a $3$-graph to contain isolated vertices. Zhang, Zhao and Lu \cite{zhang2} obtained that $H_{n,s}^2$ and $H_{n,s}^3$ are  the extremal graphs  when $s \leq (2n+4)/9$ and $s > (2n+4)/9$, respectively, among all the $3$-graphs without a matching of size $s$.
\begin{theorem}\cite{zhang2}
There exists $n_2 \in \mathbb{N}$ such that the following holds. Suppose that $H$ is a $3$-graph of order $n \geq n_2$ and $2\le s\le n/3$. If $\sigma_2(H) > \sigma_2(H_{n,s}^2)$ and $s \leq (2n+4)/9$ or $\sigma_2(H) > \sigma_2(H_{n,s}^3)$ and $s > (2n+4)/9$, then $H$ contains a matching of size $s$.
\end{theorem}

Since Theorems \ref{Kuhn2}  and \ref{the1}   have different extremal hypergraphs, Theorem~\ref{the1} does not imply Theorem~\ref{Kuhn2} (and Theorem~\ref{Kuhn2} does not imply Erd\H os' matching conjecture for 3-graphs either).  Zhang and Lu~\cite{Yi2} made the following conjecture.
\begin{conjecture}\cite{Yi2}\label{con2}
There exists $n_0 \in \mathbb{N}$ such that the following holds. Suppose that $H$ is a $3$-graph of order $n \geq n_0$ without isolated vertex. If $\sigma_2(H) > 2\left( \binom{n-1}{2}-\binom{n-s}{2}\right)$ and $n \geq 3s$, then $H$ contains no matching of size $s$ if and only if $H$ is a subgraph of $H_{n,s}^2$.
\end{conjecture}
Zhang and Lu~\cite{Yi2} showed that the conjecture holds when $n \geq 9s^2$. Later the same authors \cite{Yi3} proved the conjecture for $n \geq 13s$. Recently they \cite{Yilu2023} improve the bound to $n \geq 4s+7$. Conjecture \ref{con2}, if true, strengthens Theorem~\ref{Kuhn2} and actually provides a link between Ore type problems and Dirac type problems. However, we construct a counterexample to Conjecture \ref{con2} and thus disprove the conjecture.
\begin{theorem} \label{thextremegraph}
There exists a $3$-graph $H$ satisfying: (1) $ \sigma_2 (H)>  2\left( \binom{n-1}{2}-\binom{2n/3}{2}\right)$; (2) $H$ does not contain a perfect matching; (3) $H$ is not a subgraph of $H_{n,n/3}^2$.
\end{theorem}

In this paper, we also prove the following result which implies the result of Zhang, Zhao and Lu in \cite{zhang}.

\begin{theorem}\label{Yi1}
For all $\gamma >0$, there exists $n_0 \in \mathbb{N}$ such that the following holds for all integers $n\ge n_0$. If $H$ is a $3$-graph of order $n$ without isolated vertex and $\sigma_2(H) > (3/5+\gamma)n^2$, then $H$ contains a perfect matching if and only $H$ is not a subgraph of $H_{n,n/3}^2$.
\end{theorem}


This paper is organized as follows. We prove Theorem \ref{thextremegraph} in Section 2.   In Section 3, we give an outline of the proof along with some preliminary results. We prove a degree sum version absorbing lemma in Section 4. In Sections 5 and 6, we prove  Theorem \ref{Yi1}.

\vskip.2cm

\n{\bf Notation}
Given a graph $G$ and a vertex $u$ in $G$, $N_G(u)$ is the set of neighbors of $u$ in $G$.
Assume $H$ is a 3-uniform hypergraph. 
For $u,v\in V(H)$, denote $N_H(u,v)=\{w\in V(H):\{u,v,w\} \in E(H)\}$ (the subscript is often omitted when $H$ is clear from the context).
Given a vertex $v\in V(H)$ and a subset $A \subseteq V(H)$, we define \emph{the link} $L_{v}(A)=\{uw:u, w\in A\mbox{~and~}\{u,v,w\}\in E(H)\}$. When $A$ and $B$ are two disjoint subsets of $V(H)$, we let $L_{v}(A,B)=\{uw:u\in A,~ w\in B \mbox{~and~}\{u,v,w\}\in E(H)\}$. For a set $T$, we use $\binom{T}{\ell}$ to denote the set of all the $\ell$-element subsets of $T$. Note that $T$ may be a vertex set or an edge set in the context. Given a vertex subset $S$ of $H$, we call $S$ an independent set of $H$ if no two vertices of $S$ are adjacent. Let the maximum size over all the independent sets of $H$ be the independent number of $H$. Let $G_1$,$\cdots$,$G_t$ be $t$ graphs with the same vertex set. A set of pairwise
disjoint edges, one from each $G_i$, is called a rainbow matching for $G_1$,$\cdots$,$G_t$.

\vskip.2cm

\section{Proof of Theorem \ref{thextremegraph}}\label{counterexmapl}

In this section, we prove Theorem \ref{thextremegraph}.  Let $x,y$ be positive integers such that $n \geq 3x+3y+3 $ and let $H^{1,2}_{n,x,y}$ denote the 3-graph of order $n$, whose vertex set is partitioned into three sets $R$, $S$ and $T$ of size $x$, $2x+y$ and $n-3x-y$, respectively, and whose edge set consists of all the following triples: (1) one vertex in $R$ and the other two vertices in $T$; (2) one vertex in $S$ and the other two vertices in $T$;  (3) one vertex in $T$ and the other two vertices in $S$; (4) all three vertices in $T$. Obviously the size of maximum matching in $H^{1,2}_{n,x,y}$ is at most $x+y$, so  $H^{1,2}_{n,x,y}$ does not contain a perfect matching as $x+y \leq n/3-1$.  We also know that  the independent number of $H^{1,2}_{n,x,y}$ is $x+1$ and the independent number of  $H^{2}_{n,n/3}$ is $n/3+1$. Therefore  $H^{1,2}_{n,x,y}$ is not a subgraph of $H_{n,n/3}^2$  as $x \leq n/3-1$.
\begin{figure}[!htbp]
\begin{center}
\includegraphics[scale=0.15]{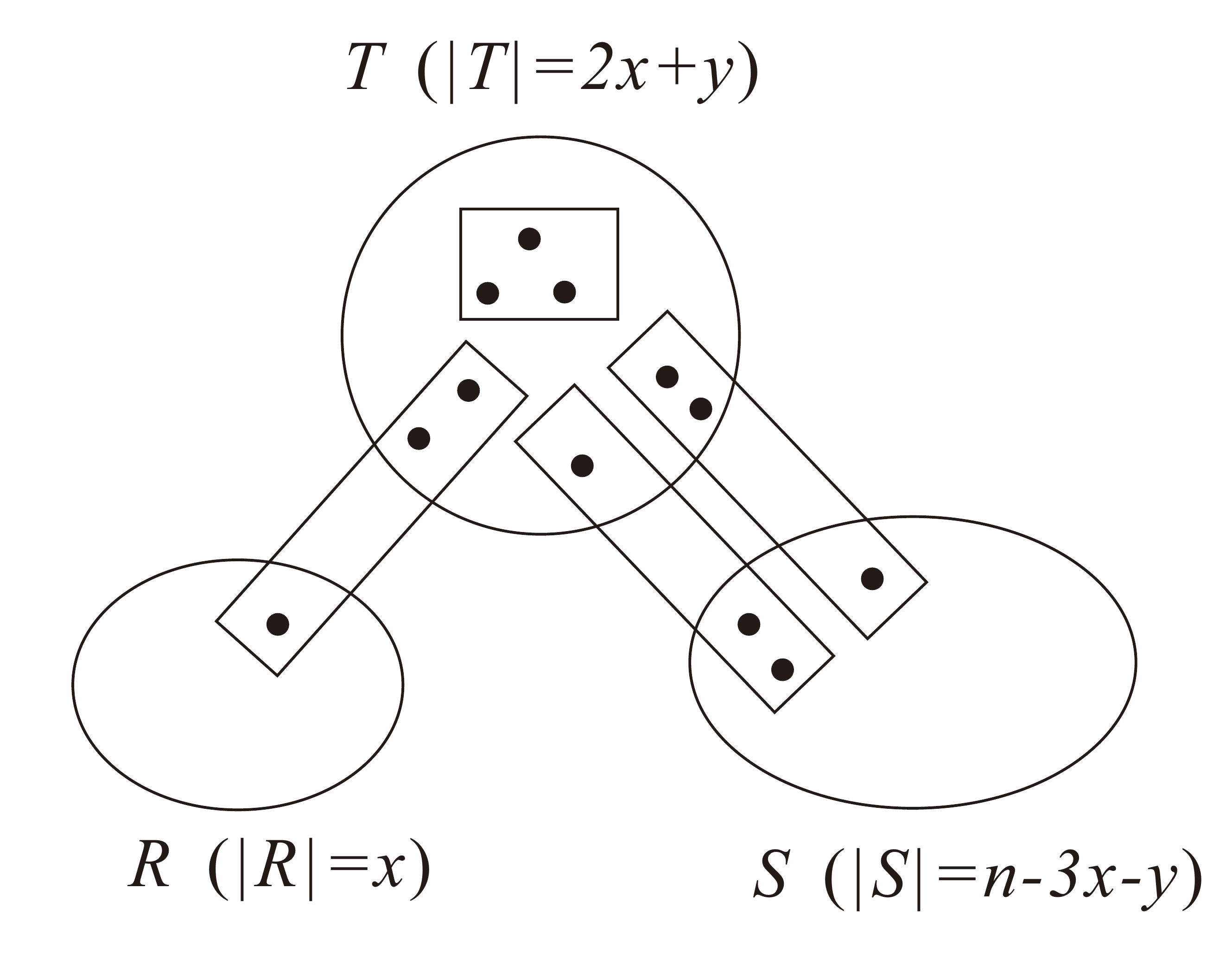}\\
\caption{The graph $H^{1,2}_{n,x,y}$.}\label{Figure1}
\end{center}
\end{figure}
The following lemma shows that we can optimize $x$ and $y$ so that $H^{1,2}_{n,x,y}$ has large vertex degree sum.
\begin{lemma} \label{extremegraph}
Let $n\in 3\mathbb{Z}$. Then
$$\max\Big\{\sigma_2(H^{1,2}_{n,x,y}): 0 \leq x,y \leq \frac{n}{3}-1,x+y \leq \frac{n}{3}-1\Big\}= \left\{
\begin{array}{ll}
(128/225)n^2-(12/5)n+2  & \mbox{if $n\equiv 0\,(\text{mod} \,5)$}\\
(128/225)n^2-(187/75)n+62/25 & \mbox{if $n\equiv 1(\text{mod} \,5)$}\\
(128/225)n^2-(184/75)n+38/25 & \mbox{if $n\equiv 2(\text{mod} \,5)$}\\
(128/225)n^2-(176/75)n+48/25 & \mbox{if $n\equiv 3(\text{mod} \,5)$}\\
(128/225)n^2-(173/75)n+42/25 & \mbox{if $n\equiv 4(\text{mod} \,5)$}.
\end{array}\right.$$
\end{lemma}
\begin{proof} For any vertex $u_1 \in R, u_2 \in S, u_3 \in T$, we obtain
$$ \deg(u_1) = \binom{2x+y}{2}, \,\, \deg(u_2) = \binom{2x+y}{2}+\binom{2x+y}{1}\binom{n-3x-y-1}{1}\,\, \text{and}$$
$$ \deg(u_3) = \binom{2x+y-1}{1}\binom{x}{1}+ \binom{2x+y-1}{2}+\binom{2x+y-1}{1}\binom{n-3x-y}{1}+\binom{n-3x-y}{2}. $$
It is not difficult to check that  $\deg(u_1) < \deg(u_2) < \deg(u_3)$. Therefore
\begin{align*}
&\max\Big\{\sigma_2(H^{1,2}_{n,x,y}): 0 \leq x,y \leq \frac{n}{3}-1,x+y \leq \frac{n}{3}-1\Big\}\\
=& \max_{0 \leq x \leq \frac{n}{3}-1} \max_{0 \leq y \leq \frac{n}{3}-1-x} \min\Big\{2\deg(u_2), \deg(u_1)+\deg(u_3)\Big\}.
\end{align*}

Furthermore
\begin{align*}
\deg(u_2) =&  \binom{2x+y}{2}+\binom{2x+y}{1}\binom{n-3x-y-1}{1}=  -\frac{1}{2}y^2+\Big(n-3x-\frac{3}{2}\Big)y+2xn-4x^2-3x.
\end{align*}
Note that the quadratic function $ -\frac{1}{2}y^2+\Big(n-3x-\frac{3}{2}\Big)y$ is maximized at $y= n-3x-\frac{3}{2}$.  Since $y \leq \frac{1}{3}n-x-1 \leq n-3x-\frac{3}{2}$, it follows that
\begin{align*}
\deg(u_2) \leq &   -\frac{1}{2}\Big(\frac{1}{3}n-x-1 \Big)^2+\Big(n-3x-\frac{3}{2}\Big)\Big(\frac{1}{3}n-x-1 \Big)+2xn-4x^2-3x \\
= & -\frac{3}{2}x^2+\Big(\frac{1}{3}n+\frac{1}{2}\Big)x+\frac{5}{18}n^2-\frac{7}{6}n+1\overset{\triangle}= f_1(x).
\end{align*}
Note that the quadratic function $ -\frac{3}{2}x^2+\Big(\frac{1}{3}n+\frac{1}{2}\Big)x$ is maximized at $x= \frac{1}{9}n+\frac{1}{6}$.

We also know
\begin{align*}
\deg(u_1)+\deg(u_3) =&  \binom{2x+y}{2}+ \binom{2x+y-1}{2}+(2x+y-1)(n-2x-y)+\binom{n-3x-y}{2} \\
 = & \frac{1}{2}y^2+\Big(3x-\frac{1}{2}\Big)y+\frac{1}{2}n^2+\frac{9}{2}x^2-xn-\frac{3}{2}n-\frac{1}{2}x+1.
\end{align*}
Note that the quadratic function $  \frac{1}{2}y^2+\Big(3x-\frac{1}{2}\Big)y$ is minimized at $y= -3x+\frac{1}{2} \leq \frac{1}{2}$.  Since $y$ is an integer, this quadratic function  increases as $y$ increases when $0 \leq y \leq n/3-x-1$ and  is maximized at $y=\frac{1}{3}n-x-1$. It follows that
\begin{align*}
\deg(u_1)+\deg(u_3) \leq & \frac{1}{2}\Big(\frac{1}{3}n-x-1 \Big)^2+\Big(3x-\frac{1}{2}\Big)\Big(\frac{1}{3}n-x-1 \Big)+\frac{1}{2}n^2+\frac{9}{2}x^2-xn-\frac{3}{2}n-\frac{1}{2}x+1
  \\
= & 2x^2-\Big(\frac{1}{3}n+2\Big)x+\frac{5}{9}n^2-2n+2\stackrel{\triangle}{=}f_2(x).
\end{align*}
Note that the quadratic function $2x^2-\Big(\frac{1}{3}n+2\Big)x$ is minimized at $x= \frac{1}{12}n+\frac{1}{2}$.

Note that for fixed $x$, if both $\deg(u_2)$ and $\deg(u_1)+\deg(u_3)$ increase as $y$ increases, then
\begin{align*}
&\max_{0 \leq x \leq \frac{n}{3}-1} \max_{0 \leq y \leq \frac{n}{3}-1-x} \min\Big\{2\deg(u_2), \deg(u_1)+\deg(u_3)\Big\} \\
=&\max_{0 \leq x \leq \frac{n}{3}-1} \min\Big\{\max_{0 \leq y \leq \frac{n}{3}-1-x}2\deg(u_2),\max_{0 \leq y \leq \frac{n}{3}-1-x}(\deg(u_1)+\deg(u_3)) \Big\}.
\end{align*}
We also know that both $\deg(u_2)$ and $\deg(u_1)+\deg(u_3)$ indeed increase as $y$ increases when $0 \leq y \leq n/3-x-1$ from the discussion above.  Therefore
\begin{align*}
&\max\Big\{\sigma_2(H^{1,2}_{n,x,y}): 0 \leq x,y \leq \frac{n}{3}-1,x+y \leq \frac{n}{3}-1\Big\}= \max_{0 \leq x \leq n/3-1}\Big\{\min\Big\{2f_1(x),f_2(x)\Big\}\Big\}.
\end{align*}
We further have
\begin{align*}
2f_1(x)-f_2(x) =& 2 \Big(-\frac{3}{2}x^2+\Big(\frac{1}{3}n+\frac{1}{2}\Big)x+\frac{5}{18}n^2-\frac{7}{6}n+1\Big)-\Big(2x^2-\Big(\frac{1}{3}n+2\Big)x+\frac{5}{9}n^2-2n+2 \Big) \\
= &  -5x^2+(n+3)x-\frac{1}{3}n.
\end{align*}
Note that the quadratic function $ -5x^2+(n+3)x$ is maximized at $x= \frac{1}{10}n+\frac{3}{10}$. We also know that $2f_1(0)-f_2(0) =-\frac{1}{3}n$, $2f_1(1)-f_2(1) =\frac{2}{3}n-2$, $2f_1(\frac{n}{5}+\frac{1}{5})-f_2(\frac{n}{5}+\frac{1}{5}) =\frac{1}{15}n+\frac{2}{5}$ and $2f_1(\frac{n}{5}+\frac{2}{5})-f_2(\frac{n}{5}+\frac{2}{5}) =-\frac{2}{15}n+\frac{2}{5}$.
Therefore
$$\max\Big\{\sigma_2(H^{1,2}_{n,x,y}): 0 \leq x,y \leq \frac{n}{3}-1,x+y \leq \frac{n}{3}-1\Big\}= \max\Big\{ 2f_1(0), f_2(1),  f_2\Big(\Big\lfloor \frac{n+1}{5} \Big\rfloor\Big), 2f_1\Big(\Big\lceil \frac{n+2}{5} \Big\rceil\Big)\Big \}.$$
It is easy to check  that
$$\max\Big\{\sigma_2(H^{1,2}_{n,x,y}): 0 \leq x,y \leq \frac{n}{3}-1,x+y \leq \frac{n}{3}-1\Big\}= \left\{
\begin{array}{ll}
f_2\Big(\lfloor \frac{n+1}{5} \rfloor\Big)  & \mbox{if $n\equiv 0\,(\text{mod} \,5)$}\\
f_2\Big(\lfloor \,\, \frac{n+1}{5} \rfloor\Big)  & \mbox{if $n\equiv 1\,(\text{mod} \,5)$}\\
2f_1\Big(\lceil \frac{n+2}{5} \rceil\Big)    & \mbox{if $n\equiv 2\,(\text{mod} \,5)$}\\
2f_1\Big(\lceil \frac{n+2}{5} \rceil\Big)   & \mbox{if $n\equiv 3\,(\text{mod} \,5)$}\\
f_2\Big(\lfloor \frac{n+1}{5} \rfloor\Big)   & \mbox{if $n\equiv 4\,(\text{mod} \,5)$}.
\end{array}\right.$$
It follows that $$\max\Big\{\sigma_2(H^{1,2}_{n,x,y}): 0 \leq x,y \leq \frac{n}{3}-1,x+y \leq \frac{n}{3}-1\Big\}= \left\{
\begin{array}{ll}
(128/225)n^2-(12/5)n+2  & \mbox{if $n\equiv 0\,(\text{mod} \,5)$}\\
(128/225)n^2-(187/75)n+62/25   & \mbox{if $n\equiv 1\,(\text{mod} \,5)$}\\
(128/225)n^2-(184/75)n+38/25    & \mbox{if $n\equiv 2\,(\text{mod} \,5)$}\\
(128/225)n^2-(176/75)n+48/25   & \mbox{if $n\equiv 3\,(\text{mod} \,5)$}\\
(128/225)n^2-(173/75)n+42/25   & \mbox{if $n\equiv 4\,(\text{mod} \,5)$}.
\end{array}\right.$$
\end{proof}

By Lemma \ref{extremegraph}, if $n\equiv 0\,(\, \text{mod} \,3)$ and $n\equiv 0\,(\text{mod} \,5)$, we let $x=\lfloor \frac{n+1}{5} \rfloor=\frac{n}{5}$ and $y=\frac{n}{3}-x-1$, then we obtain a counterexample  $H^{1,2}_{n,x,y}$ for Conjecture \ref{con2} as
$\sigma_2(H^{1,2}_{n,x,y}) = \frac{128}{225}n^2-\frac{12}{5}n+2  >  2\Big(\binom{n-1}{2}-\binom{\frac{2}{3}n}{2}\Big).$  We could obtain a similar construction for $n\equiv 0\,(\, \text{mod} \,3)$ and  $n\equiv 1,2,3,4\,(\text{mod} \,5)$. Therefore Conjecture \ref{con2} does not hold when $s=n/3$. In fact, we can make a little adjustment to $H^{1,2}_{n,x,y}$ to obtain a counterexample  of Conjecture \ref{con2} when $s$ is close to $n/3$. Motivated by this, we make  the following conjecture
\begin{conjecture}\label{newconjecture}
There exists $n_0 \in \mathbb{N}$ such that the following holds. Suppose that $H$ is a $3$-graph of order $n \geq n_0$ without isolated vertex. If $\sigma_2(H) > \max\Big\{\sigma_2(H^{1,2}_{n,x,y}): 0 \leq x,y \leq \frac{n}{3}-1,x+y \leq \frac{n}{3}-1\Big\}$, then $H$ contains no perfect matching if and only if $H$ is a subgraph of $H_{n,n/3}^2$.
\end{conjecture}

\section{Outline of the proof of Theorem \ref{Yi1}  and preliminaries}

For $0 < \gamma \ll 1$ and let $n \in 3 \mathbb{Z}$ be sufficiently large. Suppose $H$ is a $3$-graph of order $n$ without isolated vertex and $\sigma_2(H) > (3/5+\gamma)n^2$. We divide the vertex set into 2 parts $U$ and $W$.  Let $U = \{u\in V(H): \deg(u) > (3/10+\gamma/2)n^2\}$ and $W = V(H)\setminus U$.  Then any two vertices of $W$ are not adjacent; Otherwise $\sigma_2(H)\le  (3/5+\gamma)n^2$, a contradiction. If $|W| \geq \frac{1}{3}n+1$, then $|U| \leq \frac{2}{3}n-1$, $H$ is a subgraph of $H_{n,n/3}^2$ and we are done.  We thus assume that $|W| \leq \frac{1}{3}n$.  Let $$1  \gg \tau  \gg \gamma \gg \varepsilon \gg \rho \gg 1/C \gg 1/n_0 > 0.$$ We separate the cases when $(1/3-\tau)n \leq |W| \leq n/3$ and when $|W| < (1/3-\tau)n$. Namely, the following two lemmas.
\begin{lemma} \label{lemma11181}
Let $0 < \gamma \ll \tau \ll 1$, there exists $n_0$ such that the following holds. Suppose that $H$ is a $3$-graph of order $n\in 3 \mathbb{Z}$, $n \geq n_0$ without isolated vertex and $\sigma_2(H) > (3/5+\gamma)n^2$.
Let $U =\{u \in V(H): \deg(u) > \big(3/10+\gamma/2\big)n^2\}$ and $W=V(H)\setminus U$. If $(1/3-\tau)n \leq |W| \leq n/3$, then $H$ contains a perfect matching.
\end{lemma}
\begin{lemma} \label{lemma11182}
Let $0 < \gamma \ll \tau \ll 1$, there exists $n_0$ such that the following holds. Suppose that $H$ is a $3$-graph of order $n\in 3 \mathbb{Z}$, $n \geq n_0$ without isolated vertex and $\sigma_2(H) > (3/5+\gamma)n^2$.
Let $U =\{u \in V(H): \deg(u) > \big(3/10+\gamma/2\big)n^2 \}$ and $W=V(H)\setminus U$. If $ |W| \leq (1/3-\tau)n$, then $H$ contains a perfect matching.
\end{lemma}
When $(1/3-\tau)n \leq |W| \leq n/3$, we are able to significantly enhance the low bound of degree of each vertex in $U$ and in  $W$. This enables us to first find a matching $M_1$ that covers all the low-degree vertices of $W$, then, the larger degrees of the vertices in $U$ can also help us find a perfect matching. When $|W| < (1/3-\tau)n$, we need the following absorbing lemma which can be proved by the probabilistic  method.
\begin{lemma}\label{absorbinglemma1}
Let $k \geq 3$, $\varepsilon $ be any  positive number and $n$ be sufficiently large.  Suppose  $H$ is a $k$-uniform hypergraph of order $n$ with  no isolated vertex. If  $\deg(u)+\deg(v) > (1+2\varepsilon)\binom{n}{k-1}$ for any two vertices $u$ and $v$ that are contained in some edge of $H$, then there exists a matching $M$ in $H$ of size $|M| \leq \varepsilon^{4k+1}n$ such that for any vertex set $V' \subseteq V(H)$ of size at most $\varepsilon^{8k+2}n$ and $|V'| \equiv 0\,(\, \text{mod} \, k \,)$, there exists a matching $M'$ such that $V(M')=V(M) \cup V'$.
\end{lemma}
By Lemma \ref{absorbinglemma1}, we can let $M_1$ be a matching of size $\leq \varepsilon^{13} n$ such that for any vertex set $V' \in V(H)$ of size at most $\varepsilon^{26}n$  and $ |V'| \in 3\mathbb{Z}$, $H[V(M_1) \cup V']$ spans a perfect matching. Let $H'=H-V(M_1)$. We divide all the vertices of $V(H')$ into the following four groups:
$A_1=\{ u\in V(H'): \deg_{H'}(u) \leq \big(\frac{4}{15}+\gamma \big)n^2 \}$;  $A_2=\{ u\in V(H'): \big(\frac{4}{15}+\gamma \big)n^2 < \deg_{H'}(u) \leq \big(\frac{3}{10}+\frac{1}{2}\gamma- \frac{3}{2}\varepsilon^{13}\big) n^2\}$; $A_3=\{ u\in V(H'): \big(\frac{3}{10}+\frac{1}{2}\gamma- \frac{3}{2}\varepsilon^{13}\big)n^2 < \deg_{H'}(u) \leq \frac{1}{3}n^2-3\varepsilon^{13} n^2 \}$; $A_4=\{ u\in V(H'): \deg_{H'}(u) > \frac{1}{3}n^2-3\varepsilon^{13} n^2\}$. First the vertex degree sum condition can ensure that we find a matching covering all the  vertices of $A_1$. On this basis, through structural analysis, we can find a matching $M_2$ covering all vertices except at most $C$ vertices, where $C$ is a sufficiently large constant. Finally the matching $M_1$ absorbs the leftover vertices  to obtain a matching $M_3$, thus $ M_2 \cup M_3$ is a perfect matching.

\medskip
We need the following lemmas in our proof. Lemma \ref{lemma1} is Observation 1.8 of Aharoni and Howard \cite{Aha}. Lemma \ref{Lemma33} is the Chernoff bounds, which can be found in \cite{Mitzenmacher}.
A $k$-graph $H$ is called \emph{$k$-partite} if $V(H)$ can be partitioned into $V_1,\ldots,V_k$, such that each edge of $H$ meets every $V_i$ in precisely one vertex. If all parts are of the same size  $n$, we call $H$ \emph{$n$-balanced}.

\begin{lemma}\cite{Aha}
\label{lemma1}
Let $F$ be the edge set of an $n$-balanced $k$-partite $k$-graph. If $F$ does not contain $s$ disjoint edges, then $|F| \leq (s-1)n^{k-1}$.
\end{lemma}

\begin{lemma}\label{Lemma33}\cite{Mitzenmacher} Suppose $X_1,\ldots,X_n$ are independent random variables taking values in $\{0,1\}$. Let $X$ denote their sum and $\mu=\mathbb{E}[X]$ denote the expected value of $X$. Then for any $0 < \delta < 1$,
\begin{align*}
 \mathbb{P}[ X \geq (1+\delta)\mu]< e^{-\delta^2\mu/3} \ \ \ \text{and}  \ \ \
\mathbb{P} [ X \leq (1-\delta)\mu]< e^{-\delta^2\mu/2},
\end{align*}
 and for any $ \delta \geq 1$, $$ \mathbb{P} [ X \geq (1+\delta)\mu]< e^{-\delta \mu/3}.$$
\end{lemma}


The following two lemmas are Lemmas 6 and 7  in \cite{zhang}. Both of them give an upper bound of degree sum of three vertices in three graphs with some specific intersecting structure.
\begin{lemma}\cite{zhang}\label{lemmaa2}
Let $G_1, G_2, G_3$ be three graphs on the same set $V$ of $n\ge 4$ vertices such that every edge of $G_1$ intersects every edge of $G_i$ for both $i=2, 3$. Then $\sum_{i=1}^3\sum_{v\in A} \deg_{G_i}(v) \leq 6(n-1)$ for any set $A\subset V$ of size $3$.
\end{lemma}

\begin{lemma}\cite{zhang}\label{lemma3}
Let $G_1, G_2, G_3$ be three graphs on the same set $V$ of $n\ge 5$ vertices such that for any $i\ne j$, every edge of $G_i$ intersects every edge of $G_j$.
Then $\sum_{i=1}^3\sum_{v\in A} \deg_{G_i}(v) \leq 3(n+1)$ for any set $A\subset V$ of size $3$.
\end{lemma}
The following lemma is Lemma 16 in \cite{Yilu2023}.
\begin{lemma}\label{Lemma4}\cite{Yilu2023}
Given two disjoint vertex sets $A=\{u_1,u_2,\ldots,u_a\}$ and $B=\{v_1,v_2,\ldots,v_b\}$ with $ a \geq 2$ and $ b\geq 1$. Let $G_i$ be a graph on $A\cup B$ ($i=1,2,3$). Suppose that every vertex of $B$ is isolated vertex in $G_1$ and every edge of $G_i$ contains at least one vertex of $A$ for both $i=2,3$. If every edge of $G_1$ intersects every edge of $G_j$  containing one vertex from $B$ for both $j=2,3$ and every edge of $G_2$ containing one vertex from $B$ intersects every edge of $G_3$  containing one vertex from $B$, then
 \[\sum_{i=1}^3\left(\sum_{j=1}^2 \deg_{G_i}(u_j)+ \deg_{G_i}(v_1)\right) \leq \max\{6a+2,5a+2b+2\}.\]
\end{lemma}
Similarly,  we prove the following lemma.

\begin{lemma}\label{Lemma4444444444}
Given two disjoint vertex sets $A=\{u_1,u_2,\ldots,u_a\}$ and $B=\{v_1,v_2,\ldots,v_b\}$ with $ a \geq 2$ and $ b\geq 1$. Let $G_i$ be a graph on $A\cup B$ ($i=1,2,3$). Suppose that every vertex of $B$ is isolated vertex in $G_1$ and every edge of $G_i$ contains at least one vertex of $A$ for both $i=2,3$. If every edge of $G_1$ intersects every edge of $G_j$  containing one vertex from $B$ for both $j=2,3$ and every edge of $G_2$ containing one vertex from $B$ intersects every edge of $G_3$  containing one vertex from $B$, then
 \[\sum_{i=1}^3\left(\sum_{j=1}^2 \deg_{G_i}(u_j)+ 2\deg_{G_i}(v_1)\right) \leq \max\{8a+2,6a+2b+4\}.\]
\end{lemma}
\begin{proof} Let
\[y:=\sum_{i=1}^3\left(\sum_{j=1}^2 \deg_{G_i}(u_j)+ 2\deg_{G_i}(v_1)\right) \text{ and} \]
 \[ z: =\sum_{j=1}^2 \text{deg}_{G_1}(u_j) +\sum_{i=2}^3 \left( 2|N_{G_i}(v_1) \cap A| + \sum_{j=1}^2  |N_{G_i}(u_j) \cap B| \right).\]
Then
\[y=z+\sum_{i=2}^3 \left(\sum_{j=1}^2 |N_{G_i}(u_j) \cap A| \right). \]
Trivially,
\[\sum_{i=2}^3 \left(\sum_{j=1}^2 |N_{G_i}(u_j) \cap A| \right)\leq 4(a-1) .\]
We only need to prove that  $z \leq \max \{4a+6,2a+2b+8 \}$.

Let $L_1=\{u\in \{u_1,u_2\}:\deg_{G_1}(u)\ge 2\}$, $L_i=\{u\in \{u_1,u_2\}:|N_{G_i}(u) \cap B|\ge 2\} \cup \{v_1: |N_{G_i}(v_1) \cap A| \geq 3\} $, $i=2,3$  and $\ell_i=|L_i|$ for $1\le i\le 3$. We distinguish the following two cases.

\noindent{\bf Case 1. } $\ell_1 \geq 1$.

If $\ell_1 = 2$, then $N_{G_i}(u_j) \cap B = \emptyset$ and $N_{G_i}(v_1) \cap A = \emptyset$ for $i=2,3$, $j=1,2$; Otherwise we can find two disjoint edges, one from $G_1$ and the other from $G_2$ or $G_3$ containing one vertex from $B$. Therefore, $z \leq \sum_{j=1}^2\deg_{G_1}(u_j) \leq 2(a-1)$.

If $\ell_1 = 1$, say $u_1\in L_1$, then $N_{G_i}(u_2) \cap B= \emptyset$ and $N_{G_i}(v_1) \cap A \subseteq  \{u_1\}$ for $i=2,3$; Otherwise one edge of $G_1$ must be disjoint from one edge in $G_2$ or $G_3$ that contains one vertex from $B$. In this case we have $\deg_{G_1}(u_1)+\deg_{G_1}(u_2)  \leq a-1+1$ and $2|N_{G_i}(v_1)\cap A|+\sum_{j=1}^2|N_{G_i}(u_j)\cap B|  \leq 2+b$ for $i=2,3$. Therefore, $z\leq a+2b+4$.

\noindent{\bf Case 2. } $\ell_1 = 0$.

In this case $\sum_{j=1}^2\deg_{G_1}(u_j) \leq 2$. Without loss of generality, we assume that $\ell_2 \geq \ell_3$. First, if $\ell_2=3$, then we have $\text{deg}_{G_1}(u_j) =0$ and $N_{G_3}(u_j) \cap B = N_{G_3}(v_1) \cap A = \emptyset$ for both $j=1,2$. Consequently,  $z \leq 2a+2b$.
Secondly, we assume that $\ell_2 \leq 1$,  then $\sum_{j=1}^2|N_{G_i}(u_j)\cap B| +2|N_{G_i}(v_1)\cap A| \leq \max\{2a+2,b+5\}$ for $i=2,3$. It follows that $z \leq \max\{4a+6,2a+2b+8\}$ as $a \geq 2$ and $b \geq 1$. Thirdly, we assume that $\ell_2 = 2$. If $|N_{G_2}(v_1)\cap A| \geq 3$ and $|N_{G_2}(u_1)\cap B| \geq 2$ (for $|N_{G_2}(u_2)\cap B| \geq 2$, similarly), then we have that $N_{G_3}(u_1) \cap  B \subseteq \{v_1\}$,   $N_{G_3}(u_2) \cap B= \emptyset$ and $N_{G_3}(v_1) \cap  A =\{u_1\}$; Otherwise one edge from $G_2$ is disjoint from one edge from $G_3$, and both of these two edges contain one vertex from $B$.  Therefore, $\sum_{j=1}^2|N_{G_2}(u_j)\cap B| +2|N_{G_2}(v_1)\cap A| \leq 2a+b+1$ and $\sum_{j=1}^2|N_{G_3}(u_j)\cap B| +2|N_{G_3}(v_1)\cap A| \leq 3$, which implies that $z \leq 2a+b+6$.  If $|N_{G_2}(u_1)\cap B| \geq 2$ and $|N_{G_2}(u_2)\cap B| \geq 2$, then we have that $N_{G_3}(u_j) \cap  B =\emptyset$ for both $j=1,2$ and $N_{G_3}(v_1) \cap  A =\emptyset$.  Therefore, $\sum_{j=1}^2|N_{G_2}(u_j)\cap B| +2|N_{G_2}(v_1)\cap A| \leq 2b+4$, which implies that $z \leq 2b+6$.
\end{proof}

\begin{lemma}\label{Yilu}\cite{Yilu2023}
Let $H=(V_1\cup V_2\cdots V_k, E)$ be an $k$-balanced $k$-partite $k$-graph, where $V_i=\{u_1^i,u_2^i,\cdots,$ $u_{k-1}^i,v^i\}$, $i\in [k]$ and no edge of $E$ contains two vertices $v^i$ and $v^j$ $(i \neq j)$. If $H$ contains no perfect matching, then $|E| \leq 2 (k-1)^k$.
\end{lemma}


\begin{lemma}\label{lemma120645778}
Let $H=(V_1\cup V_2\cup V_3, E)$ be a $3$-balanced $3$-partite $3$-graph, where $V_i=\{u_1^i,u_2^i,v^i\}$, $i\in [3]$ and no edge of $E$ contains two vertices $v^i$ and $v^j$ $(i \neq j)$. If $H$ contains no perfect matching, then $2d(v^3)+d(u_1^3)+d(u_2^3) \leq 20$.
\end{lemma}
\begin{proof} By the definition of edge set $E$, we have $d(v^3) \leq 4$, $d(u_1^3) \leq 8$ and $d(u_2^3) \leq 8$. If $d(v^3) \leq 2$, the result holds obviously. Assume that $d(v^3) = 3$, by symmetry, let $\{u_1^1,u_1^2,v^3\},\{u_1^1,u_2^2,v^3\},\{u_2^1,u_1^2,v^3\} \in E(H)$. At least one of $\{u_2^1,v^2,u_1^3\}$ and $\{v^1,u_2^2,u_2^3\}$ does not belong to $E(H)$; Otherwise $H$ contains a perfect matching, a contradiction. Similarly at least one of $\{u_2^1,v^2,u_2^3\}$ and $\{v^1,u_2^2,u_1^3\}$ does not belong to $E(H)$. Therefore $2d(v^3)+d(u_1^3)+d(u_2^3) \leq 20$. The case  for $d(v^3) = 4$ can be shown similarly.
\end{proof}

When we prove Lemma \ref{lemma11182}, we need the following fact and  two lemmas. Specifically,  these two lemmas have been adapted from Claims 2 and 3 in  \cite{Kuhn2}.
\begin{fact}\label{fact1}\cite{Kuhn2}
Let $B$ be a balanced bipartite graph on $6$ vertices.\\
$\bullet$ If $e(B) \geq 7$ then $B$ contains a perfect matching.\\
$\bullet$ If $e(B) = 6$ then either $B$ contains a perfect matching or $B \cong B_{033}$.\\
$\bullet$ If $e(B) = 5$ then either $B$ contains a perfect matching or $B \cong B_{023}, B_{113}$.
\end{fact}
\begin{figure}[!htbp]
\begin{center}
\includegraphics[scale=0.20]{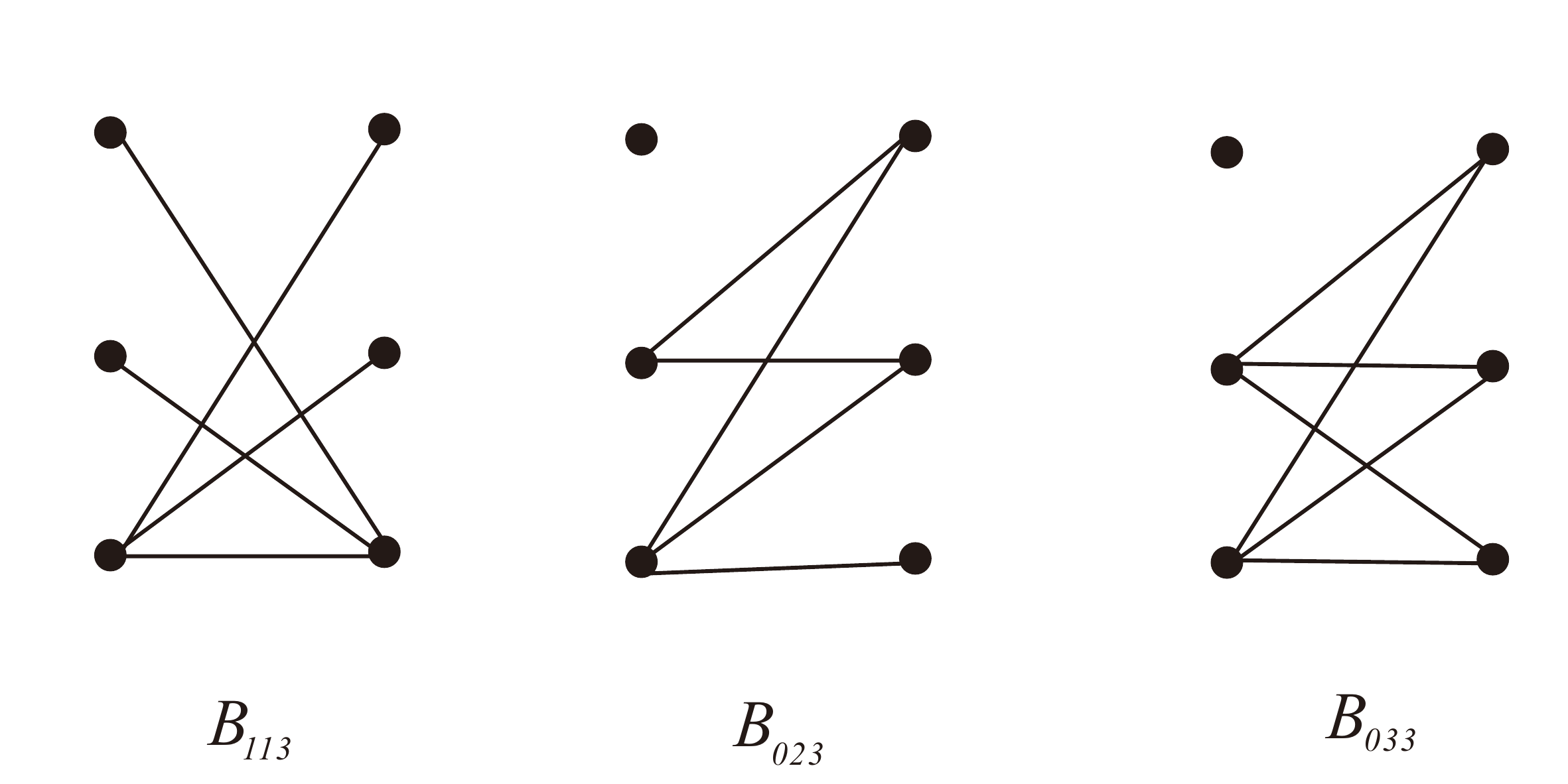}\\
\caption{The graphs $B_{113}$, $B_{023}$ and $B_{033}$.}\label{Figure1}
\end{center}
\end{figure}

\begin{lemma}\cite{Kuhn2}\label{perfectmatching}
Let $0 < 1/n_0 \ll 1/C \ll \rho \ll 1$. Let $M$  be a maximum matching of a 3-graph $H$ of order $n \geq n_0$ and $V_0= V(H)\setminus V(M)$.  Let $V_0'$ denote the set of all those vertices $u\in V_0$ for which there are at least $\rho n^2$ pairs $EF\in \binom{M}{2}$ such
that $L_u(EF)$ contains a perfect matching. Then $|V_0'| \leq C$.
\end{lemma}
\begin{lemma}\cite{Kuhn2}\label{b023b033}
Let $0 < 1/n_0 \ll 1/C \ll \rho \ll 1$. Let $M$  be a maximum matching of a 3-graph $H$ of order $n\geq n_0$ and $V_0= V(H)\setminus V(M)$.  Let $V_0''$ denote the set of all those vertices $u\in V_0$ for which there are at least $\rho n^2$ pairs $EF\in \binom{M}{2}$ such
that $L_u(EF) \cong B_{023}, B_{033}$. Then $|V_0''| \leq C$.
\end{lemma}

\section{Absorbing lemma \ref{absorbinglemma1}}

{ \noindent \bf Proof of Lemma  \ref{absorbinglemma1}.}  Let $W' = \{w \in V(H): d(w) \leq (\frac{1}{2}+\varepsilon)\binom{n}{k-1} \}$ and $U' = V(H)\setminus W'$. Clearly no edge contains two vertices from $W'$. For  $A \subseteq \binom{V(H)}{k}$, we call a set $T \in \binom{V}{2k^2}$ of size $2k^2$ an absorbing $2k^2$-set for $A$ if both $H[T]$ and $H[A\cup T]$ contain a perfect matching.  We have the following claim.
\begin{claim}\label{claim32}
For every $A \subseteq \binom{V(H)}{k}$, there are at least  $ \varepsilon^{4k} \binom{n}{k-1}^{2k} n^{2k}$ absorbing $2k^2$-sets for $A$.
\end{claim}

\begin{proof}Let $A=\{u_1,u_2,\cdots,u_k\}$.  Since $H$ does not contain isolated vertices, there exists $v_i$ adjacent to $u_i$ for $i \in [k]$. Clearly $\deg(u_i)+\deg(v_i) > (1+2\varepsilon)\binom{n}{k-1}$ and $ \deg(v_i) \leq \binom{n-1}{k-1}$. Therefore  $\deg(u_i)> 2\varepsilon\binom{n}{k-1}$ for any $i \in \{1,2,\cdots,k\}$. This implies that we can select at least $ \varepsilon n$ vertices $u_{k+i}\in U'\setminus \{u_1,\cdots,u_{k+i-1}\}$ adjacent to $u_{i}$ for $i \in [k]$. Otherwise if $u_i\in W'$ then $\deg(u_i) \leq \binom{\varepsilon n+2k}{k-1} < 2\varepsilon\binom{n}{k-1}$; if $u_i\in U'$ then $\deg(u_i) \leq |W|\binom{\varepsilon n+2k}{k-2}+\binom{\varepsilon n+2k}{k-1} < (\frac{1}{2}+\varepsilon)\binom{n}{k-1}$, a contradiction.
\begin{figure}[!htbp]
\begin{center}
\includegraphics[scale=0.18]{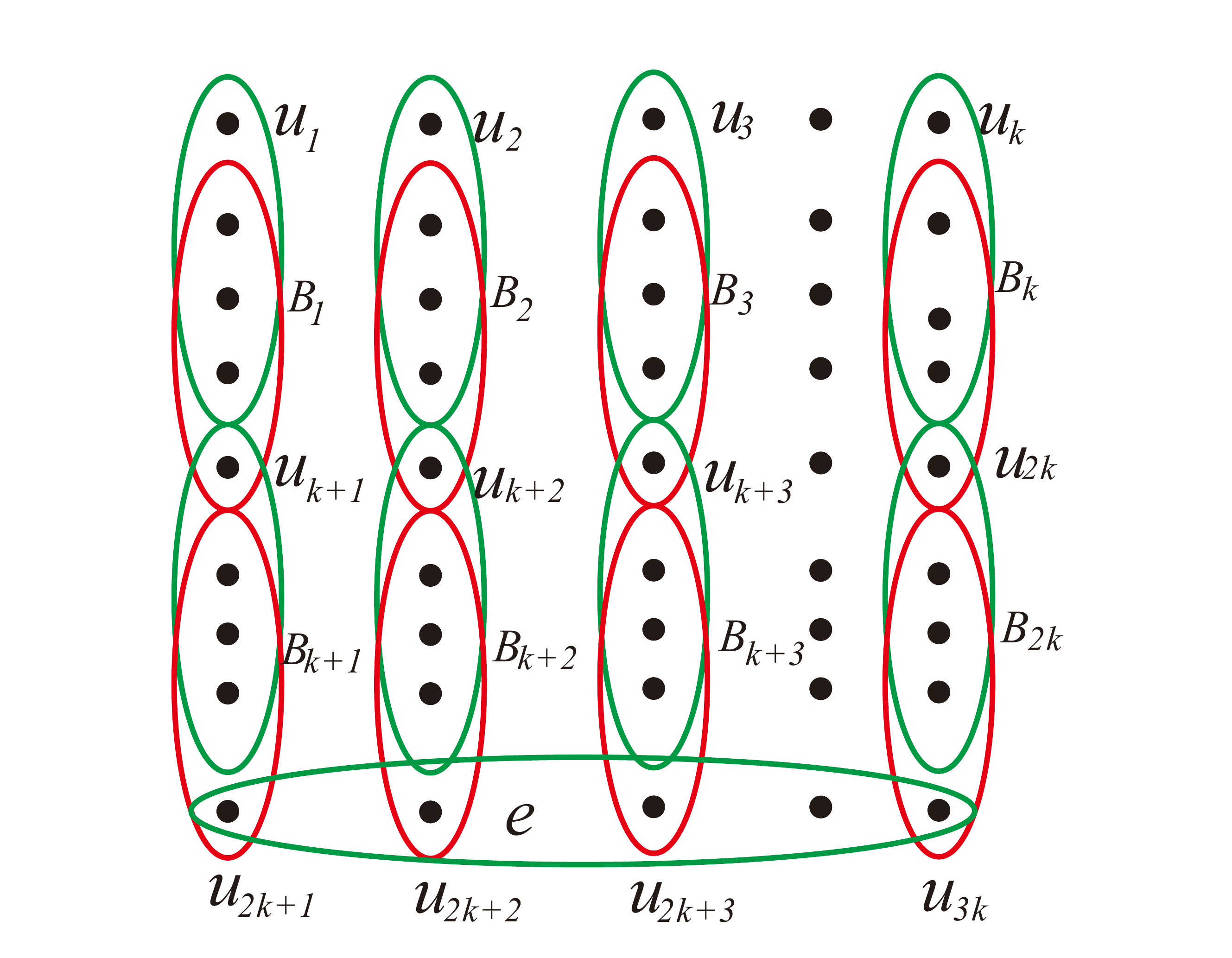}\\
\caption{An absorbing $2k^2$-sets $T$ for $A$ ($T$ is red and$H[A \cup T]$ is green).}
\end{center}
\end{figure}
Next we prove that there are at least $ \varepsilon^2\binom{n}{k-1}n/(k+1)$ edges $e \subseteq U'$ disjoint with $\{u_1,\cdots,u_{2k}\}$.
For each vertex $u \in U'$, $d(u) > (\frac{1}{2}+\varepsilon)\binom{n}{k-1}$.  Since $n$ is sufficiently large, it is not difficult to check that there are at most $|W'| \binom{|U'|}{k-2} < \frac{1}{2}\binom{n}{k-1}$ edges consisting of $u$ and one vertex of $W'$. Consequently $H$ contains at least $\varepsilon\binom{n}{k-1}$ edges $e \subseteq U'$ that include $u$.  Moreover $|U'| > \varepsilon n$; Otherwise $\deg(u) \leq |W'|\binom{\varepsilon n}{k-2}+\binom{\varepsilon n}{k-1} < \frac{1}{2}\binom{n}{k-1}$ for $u\in U'$, a contradiction. We further derive that the number of edges $e \subseteq U'$ is greater than $ \varepsilon^2\binom{n}{k-1}n/k $. Therefore there are at least  $ \varepsilon^2\binom{n}{k-1}n/(k+1)$ edges $e \subseteq U'$ disjoint with $\{u_1,\cdots,u_{2k}\}$ as $ \varepsilon^2\binom{n}{k-1}n/k -\binom{2k}{1}\binom{n}{k-1}> \varepsilon^2\binom{n}{k-1}n/(k+1).$ Let  $e=\{u_{2k+1},\cdots,u_{3k}\}$ be one such edge. By the choice of $u_{k+1},\cdots,u_{3k}$,
we obtain $\deg(u_{i})+\deg(u_{k+i}) > (1+2\varepsilon)\binom{n}{k-1}$ for $i \in [2k]$.  This implies that we can select $\varepsilon\binom{n}{k-1}$ sets $B_i$ such that $B_i$ is disjoint to $ \bigcup _{j\in[i-1]}B_{j} \cup \{u_1,\cdots,u_{3k}\}$ and both $B_i \cup \{u_i\}$ and $B_i \cup \{u_{k+i}\}$ form edges in $H$. Indeed, because there are at most $2(2k^2+k)\binom{n}{k-2}$ edges consisting of $u_i$ or $u_{k+i}$  and one vertex in $ \bigcup _{j\in[i-1]}B_{j} \cup \{u_1,\cdots,u_{3k}\}$, and  $2\varepsilon\binom{n}{k-1}- 2(2k^2+k)\binom{n}{k-2} > \varepsilon\binom{n}{k-1}$. Thus there are at least  $(\varepsilon n)^k (\varepsilon^2\binom{n}{k-1}n/(k+1))(\varepsilon \binom{n}{k-1})^{2k} \geq \varepsilon^{4k} \binom{n}{k-1}^{2k} n^{2k}$ absorbing $2k^2$-sets for $A$.
\end{proof}

Let $\mathcal{L}(A)$ denote the family of all those $2k^2$-sets absorbing $A$. By Claim \ref{claim32}, we know $|\mathcal{L}(A)| \geq  \varepsilon^{4k} \binom{n}{k-1}^{2k} n^{2k}$. Let $\mathcal{F}$ be a family of 2$k^2$-sets by selecting each of the $\binom{n}{2k^2}$ possible $2k^2$-sets independently with probability
\begin{align*}
 p = \frac{\varepsilon^{4k+1} n}{\triangle} \ \ \ \text{with}  \ \ \
\triangle =2 \binom{n}{k-1}^{2k} n^{2k}  \geq 2n \binom{n}{2k^2-1} \geq 4k^2 \binom{n}{2k^2}.
\end{align*}
Clearly
\begin{align*}
 \mathbb{E} |\mathcal{F}| \leq \frac{\varepsilon^{4k+1} n}{4k^2} \ \ \ \text{and}  \ \ \   \mathbb{E} |\mathcal{L}(A) \cap \mathcal{F}| \geq \frac{ \varepsilon^{8k+1} n}{2} \ \ \ \forall \ A \in \binom{V}{k}.
\end{align*}
By Lemma \ref{Lemma33}, with probability $1-o(1)$, as $n \rightarrow \infty$ the family $ \mathcal{F}$ has the following properties:
\begin{align}\label{1}
 |\mathcal{F}| \leq \frac{\varepsilon^{4k+1} n}{3k^2}  \ \ \ \text{and}  \ \ \ |\mathcal{L}(A) \cap \mathcal{F}| \geq \frac{\varepsilon^{8k+1} n}{3}, \ \ \ \forall \ A \in \binom{V}{k}.
\end{align}
Furthermore, we can bound the expected number of intersecting $2k^2$ sets by
\begin{align*}
 \binom{n}{2k^2} \times 2k^2 \times \binom{n}{2k^2-1} \times p^2 \leq \frac{\varepsilon^{8k+2}n}{4}.
\end{align*}
By Markov's inequality,
\begin{align*}
 P\Big(\text{$\mathcal{F}$ contains at least $\varepsilon^{8k+2}n$ intersecting pairs} \Big) \leq \frac{\varepsilon^{8k+2}n/4}{\varepsilon^{8k+2}n} =\frac{1}{4}.
\end{align*}
As a result, with probability at least $3/4$,
\begin{align}\label{3}
 \text{$\mathcal{F}$ contains at most $\varepsilon^{8k+2}n$ intersecting pairs.}
\end{align}

It follows that the family $\mathcal{F}$ possesses all the properties stated in (\ref{1}) and (\ref{3})  with positive probability. Remove all the intersecting and non-absorbing $2k^2$-sets from $\mathcal{F}$, we obtain a subfamily $\mathcal{F'}$ of size $|\mathcal{F'}| \leq \frac{\varepsilon^{4k+1} n}{3k^2} $ consisting of pairwise disjoint absorbing $2k^2$-sets which satisfies
 \begin{align*}
 |\mathcal{L}(A) \cap \mathcal{F'}| \geq \frac{ \varepsilon^{8k+1} n}{2}-\varepsilon^{8k+2}n\geq\varepsilon^{8k+2}n, \ \ \ \forall \ A \in \binom{V}{k}.
\end{align*}

Because $\mathcal{F'}$ consists of pairwise disjoint absorbing $2k^2$-sets, $H[V(\mathcal{F'})]$ contains a
perfect matching $M$ of size at most $2\varepsilon^{4k+1}n/3k$. We can partition $V'$ into $|V'|/k$ $k$-sets.  Since $|V'|/k \leq \varepsilon^{8k+2}n$, we
can greedily absorb these $|V'|/k$ parts  using $|V'|/k$ different absorbing $2k^2$-sets respectively. Therefore there
exists a matching covering exactly the vertices in $V(\mathcal{F'}) \cup V'$.
\qed

\section{Proof of Lemma \ref{lemma11181}}

Since $n/3 \geq|W| \geq n/3- \tau n$, we have  $2n/3 \leq |U| \leq 2n/3+ \tau n$.   It follows that $d(w) \leq  \binom{|U|}{2} \leq \binom{2n/3+ \tau n}{2}$ and $d(u) \leq  \binom{n}{2}-\binom{|W|}{2} \leq \binom{n}{2}-\binom{n/3- \tau n}{2}$ for any vertex $w \in W, u \in U$. Every vertex $w \in W$ is adjacent to one vertex $u \in U$, we obtain that $d(w) > \big(\frac{3}{5}+\gamma\big)n^2- (\binom{n}{2}-\binom{n/3- \tau n}{2})\geq \big(\frac{7}{45}+\gamma-\tau\big) n^2$.  Furthermore, we claim that each vertex $u\in U$ is adjacent to one vertex of $W$. Otherwise $d(u) \leq  \binom{|U|}{2} \leq \binom{2n/3+ \tau n}{2} < (\frac{3}{10}+\frac{1}{2}\gamma)n^2$, a contradiction. Therefore  $$d(u) > \Big(\frac{3}{5}+\gamma\Big)n^2-d(w) \geq \Big(\frac{3}{5}+\gamma\Big)n^2- \binom{|U|}{2} \geq \Big(\frac{3}{5}+\gamma\Big)n^2- \binom{2n/3+ \tau n}{2}\geq \Big(\frac{17}{45}+\gamma-\tau \Big) n^2.$$  First, we have the following claim.
\begin{claim}\label{lemma24}
$H$ contains a matching covering every vertex of $W$.
\end{claim}
\begin{proof}
Let $M$ be a largest matching such that each edge $e \in M$ contains one vertex of $W$. Set $w \in W\setminus V(M)$ and $u_1,u_2 \in U\setminus V(M)$.  We have $d(u_1)+d(u_2)+d(w) \geq \big(41/45+3\gamma-3\tau\big) n^2$.

 Obviously, $H$ contains at most $7|M|$ edges consisting of one vertex from $Q=\{u_1,u_2,w\}$ and two vertices from the same edge of $M$, at most $2 \binom{|U\setminus V(M)|}{2}$ edges with one vertex from $\{u_1,u_2,w\}$ and two other vertices  from $U\setminus V(M)$, and no edge with one vertex from $Q$, another vertex from $U\setminus V(M)$ and  the third vertex from $W\setminus V(M)$.
\begin{subclaim}\label{claim180}
For any two distinct edges $e_1$, $e_2$ from $M$, we have $\sum _{u \in Q} |L_{u}(e_1,e_2)| \leq 16$.
\end{subclaim}
\begin{proof} Let $H_1$ be the 3-partite subgraph of $H$ induced on three parts $V,e_1$ and $e_2$. We observe that $H_1$ does not contain a perfect matching. Otherwise, we obtain a larger matching than $M$, a contradiction. Applying Lemma \ref{Yilu} with $k=3$, we obtain that $|E(H_1)| \leq 16$. Therefore  $\sum _{u \in Q} |L_{u}(e_1,e_2)| \leq 16$.
\end{proof}
\begin{subclaim}\label{claim190}
For any $e\in M$,   we have
$\sum _{u\in Q} |L_{u}(e,V(H)\setminus V(M))| \leq \max\{6|U \setminus V(M)|+2,5|U \setminus V(M)|+2|W \setminus V(M)|+2\}.$
\end{subclaim}
\begin{proof} Let $e=\{w',u_2',u_3'\} \in M $ with $w' \in W$. Apply Lemma \ref{Lemma4} with $A=U \setminus V(M)$, $B=W \setminus V(M)$, $G_1 = L_{w'}((U \cup W) \setminus V(M))$ and $G_i = L_{u_i'}((U \cup W) \setminus V(M))$ for $i=2,3$. Recall that $|U \setminus V(M)| \geq 2$ and $|W \setminus V(M)| \geq 1$. Every edge of $G_1$ intersects every edge of $G_j$  containing one vertex in $B$ for $j=2,3$, and every edge of $G_2$ containing one vertex in $B$ intersects every edge of $G_3$  containing one vertex in $B$. Otherwise, we obtain a larger matching than $M$, a contradiction.  By Lemma \ref{Lemma4}, we obtain $\sum _{u\in Q} |L_{u}(e,U \setminus V(M))| \leq \max\{6|U \setminus V(M)|+2,5|U \setminus V(M)|+2|W \setminus V(M)|+2\}.$
\end{proof}
Note that $|L_u( U\setminus V(M),W\setminus V(M))|=0$ for $u \in Q$.  Combining these bounds together, we obtain that
\begin{align*}
\sum_{u \in Q} \text{deg}(u)  \leq & 16\binom{|M|}{2}+7|M| +\sum_{u\in Q}|L_{u}(V(M),V(H) \setminus V(M))|+\sum_{u\in Q}|L_{u}(U \setminus V(M))|.
\end{align*}

Since $n/3 \geq|W| \geq n/3- \tau n$, we have $|U| \geq 2|W| $. Thus  $|U \setminus V(M)| \geq 2|W \setminus V(M)| $.   Recall that $|U \setminus V(M)|=|U|-2|M|$. We obtain
\begin{align*}
\sum_{u \in Q} \text{deg}(u)  & \leq  16\binom{|M|}{2}+7|M| +|M|(6(|U|-2|M|)+2)+2 \binom{|U|-2|M|}{2}\\
& = (2|U|+3)|M|+|U|^2-|U|.
\end{align*}
Since $|M| \leq n-|U|-1$, it follows that
\begin{align*}
\sum_{u \in Q} \text{deg}(u)  & \leq   (2|U|+3)(n-|U|-1)+|U|^2-|U|=-(|U|-n+3)^2+n^2-3n+6.
\end{align*}
Since $|U| \leq \frac{2}{3}n+ \tau n$ and  $n$ is sufficiently large, we obtain
\begin{align*}
\sum_{u \in Q} \text{deg}(u)  & \leq -\Big(\frac{2}{3}n+ \tau n-n+3\Big)^2+n^2-3n+6\\
&\leq  \Big(\frac{8}{9}+ \frac{2}{3} \tau-\tau^2\Big)n^2-(6\tau+1)n-3 \\
& < \Big(\frac{41}{45}+3\gamma-3\tau\Big) n^2.
\end{align*}
It is a contradiction.
\end{proof}

\begin{claim}
$H$ contains a perfect matching.
\end{claim}
\begin{proof} Let $M$ be a largest matching of $H$ that (i) $M$ contains all the vertices of $W$; (ii)  subject to (i), the size of $M$ is maximum. By Claim \ref{lemma24}, such $M$ does exist. Let $M_1 = \{e \in M: e \cap W \neq \emptyset \}$ and $M_2 = M\setminus M_1$.  Suppose to the contrary that $|M| < n/3$. Let $u_1,u_2,u_3 \in U\setminus V(M)$ and $Q=\{u_1,u_2,u_3\}$.  We have $\sum_{ u \in Q}d(u) \geq \big(\frac{17}{15}+3\gamma-3\tau\big) n^2$.

 Obviously, $H$ contains at most $9|M|$ edges consisting of one vertex from $Q$ and two vertices from the same edge of $M$, and no edge with one vertex from $Q$ and two other vertices  from $V(H)\setminus V(M)$.
\begin{subclaim}\label{claim180}
For any two distinct edges $e_1$, $e_2$ from $M$, we have $\sum _{u \in Q} |L_{u}(e_1,e_2)| \leq 18$.
\end{subclaim}
\begin{proof} Let $H_1$ be the 3-partite subgraph of $H$ induced on three parts $Q,e_1$ and $e_2$. We observe that $H$ does not contain a perfect matching. Otherwise, we obtain a larger matching than $M$, a contradiction. Applying Lemma \ref{lemma1} with $k=3$, we obtain that $|E(H_1)| \leq 18$. Therefore  $\sum _{u \in Q} |L_{u}(e_1,e_2)| \leq 18$.
\end{proof}

\begin{subclaim}\label{claim18}
 For any $e \in M_1$, we have $\sum_{i=1}^3|L_{u_i}(e,V(H)\setminus V(M) )| \leq 6|V(H)\setminus V(M)|$.
\end{subclaim}
\begin{proof} Assume $e =\{u_1', u_2', u_3'\} \in M_1$ with $u_1' \in W$ and $u_2', u_3'\in U$.
For $i=1,2,3$, let $G_i$ be the graph obtained from $L_{u_i'}(V(H)\setminus V(M))$ after adding an isolated vertex $u'$.  Then $|V(G_i)|=|V(H)\setminus V(M)|+1 \geq 4$. By the choice of $M$, every edge of $G_1$ intersects every edge of $G_2$ and $G_3$. The desired inequality thus follows from Lemma \ref{lemmaa2}.
\end{proof}

\begin{subclaim}\label{claim19} For any $e \in M_2$, we have $ \sum_{i=1}^3|L_{u_i}(e,V(H)\setminus V(M) )| \leq 3(|V(H)\setminus V(M) |+3)$.
\end{subclaim}
\begin{proof} Assume $e =\{u_1',u_2',u_3'\} \in M_2$ with $u_1' ,u_2',u_3'\in U$.
For $i=1,2,3$, let $G_i$ be the graph obtained from $L_{u_i'}(V(H)\setminus V(M) )$ after adding two isolated vertices $u'$ and  $u''$.  Then $|V(G_i)|= |V(H)\setminus V(M) |+2 \geq 5$. Since $M$ is maximum, the desired inequality follows from Lemma \ref{lemma3}.
\end{proof}

Note that $|L_{u}(V(H) \setminus V(M))|=0$ for $u \in Q$. Combining these bounds together, we obtain that
\begin{align*}
\sum_{u \in Q} \text{deg}(u)  \leq & 18\binom{|M|}{2}+9|M| +\sum_{u\in Q}|L_{u}(V(M),V(H) \setminus V(M))|+\sum_{u\in Q}|L_{u}(V(H) \setminus V(M))| \\
\leq & 18\binom{|M|}{2}+9|M| +6|M_1||V(H)\setminus V(M)|+3|M_2|(|V(H)\setminus V(M)|+3).
\end{align*}
Since $|M_2|=|M|-|M_1|$ and $|V(H)\setminus V(M)|=n-3|M|$, we obtain
\begin{align*}
\sum_{u \in Q} \text{deg}(u)  \leq &  18\binom{|M|}{2}+9|M| +6|M_1|(n-3|M|)+3(|M|-|M_1|)(n-3|M|+3)\\
= & (-9|M|+3n-9)|M_1|+3|M|n+9|M|.
\end{align*}
Since $|M| \leq n/3-1$ and $|M_1| \leq |M|$, we have
\begin{align*}
\sum_{u \in Q} \text{deg}(u)  \leq &   (-9|M|+3n-9)|M|+3|M|n+9|M| \\
=& -9 \Big(|M|-\frac{1}{3}n\Big)^2+n^2.
\end{align*}
Since $n/3-1 \geq |M| \geq |W|\geq n/3-\tau n$,
\begin{align*}
\sum_{u \in Q} \text{deg}(u)  & \leq    -9 \Big(\frac{1}{3}n-1-\frac{1}{3}n\Big)^2+n^2 =   n^2-9  < \Big(\frac{17}{15}+3\gamma-3\tau\Big) n^2.
\end{align*}
It is a contradiction.
\end{proof}

\section{Proof of Lemma \ref{lemma11182}}

Recall that $1  \gg \tau  \gg \gamma \gg \varepsilon \gg \rho \gg 1/C \gg 1/n_0 > 0.$  Suppose that $H$ is a $3$-graph of order $n$ without isolated vertex and $\sigma_2(H) > (3/5+\gamma)n^2$.
By Lemma \ref{absorbinglemma1}, let $M_1$ be a matching of size less than $ \varepsilon^{13} n$ such that for any vertex set $V' \in V(H)$ of size at most $\varepsilon^{26} n$, and  $|V'| \in 3\mathbb{Z}$, $H[V(M_1) \cup V']$ spans a perfect matching. Let $H'=H-V(M_1)$. It is easy to obtain that
\begin{align}
\sigma_2(H') >\Big(\frac{3}{5}+\gamma\Big)n^2- 2\binom{|V(M_1)|}{2}-2\binom{|V(M_1)|}{1}n > \Big(\frac{3}{5}+\gamma - 3\varepsilon^{13}\Big) n^2.  \label{labelGH123212}
\end{align}
We divide all the vertices of $H'$ into the following four groups:
$A_1=\{ u\in V(H'): \deg_{H'}(u) \leq \big(\frac{4}{15}+\gamma \big)n^2 \}$;  $A_2=\{ u\in V(H'): \big(\frac{4}{15}+\gamma \big)n^2 < \deg_{H'}(u) \leq \big(\frac{3}{10}+\frac{1}{2}\gamma- \frac{3}{2}\varepsilon^{13}\big) n^2\}$; $A_3=\{ u\in V(H'): \big(\frac{3}{10}+\frac{1}{2}\gamma- \frac{3}{2}\varepsilon^{13}\big)n^2 < \deg_{H'}(u) \leq \frac{1}{3}n^2-3\varepsilon^{13} n^2 \}$; $A_4=\{ u\in V(H'): \deg_{H'}(u) > \frac{1}{3}n^2-3\varepsilon^{13} n^2\}$.

Let $U' = A_2 \cup A_3 \cup A_4$ and $W' = A_1$. For any $w \in W'$, $\deg_{H'}(w) \leq \big(4/15+\gamma \big)n^2$.  Clearly no two vertices in $W'$ are adjacent. We further obtain that  $\deg_{H}(w) \leq  \deg_{H'}(w)+\binom{3|M_1|}{1}n \leq \big(3/10+\gamma/2\big)n^2$. Therefore $|W'| \leq |W|\leq n/3-\tau n$.  For any $u \in U\setminus V(M_1)$, $\deg_{H}(u) > \big(3/10+\gamma/2 \big)n^2$. We further obtain that
\begin{align}\label{4}
\deg_{H'}(u) >  \deg_{H}(u)-\binom{3|M_1|}{1}n > \Big(\frac{3}{10}+\frac{1}{2}\gamma-3\varepsilon^{13} \Big)n^2 > \Big(\frac{4}{15}+\gamma \Big)n^2.
\end{align}
Therefore $U\setminus V(M_1) \subseteq U'$. It follows that $|U'| \geq |U|-3|M_1| \geq 2n/3+\tau n-3 \varepsilon^{13}n$ as $|U| \geq 2n/3+\tau n$ and $|M_1|\leq \varepsilon^{13}n$. We further derive that $| U'| \geq 2|W'|$.



First, we show the following claim.
\begin{claim}\label{Claimcoveringalllsmaler}
$H'$ contains a matching that covers all the vertices in $W'$.
\end{claim}
\begin{proof}
Let $M$ be a largest matching that (i) each edge $e \in M$ contains one vertex of $W'$; (ii) subject to (i), $\max_{w\in W'\setminus V(M)} |\{u \in U' \setminus V(M): \text{$u$ and $w$ are adjacent} \}|$ is maximized. We assume, to the contrary, that $W'\setminus V(M) \neq \emptyset$. Let $W'_1= W'\setminus V(M)$ and $U'_1= U'\setminus V(M)$.
We distinguish the following two cases:\medskip

{\noindent  \medskip \bf Case 1.} Every vertex of $W'_1$ is adjacent to at most one vertex of $U'_1$.

We further distinguish the following two subcases:\medskip

{\noindent \bf Case 1.1} $|W'| \leq \frac{n}{\sqrt{15}}$.

Fix $w \in W'_1$. Since $w$ is adjacent to at most one vertex of $U'_1$, we have $\deg_{H'}(w) \leq \binom{|V(M) \cap U'|+1}{2}=\binom{2|M|+1}{2}$. Since $w$ is not an isolated vertex, there exists a vertex $u\in U'\cap V(M)$ that is adjacent to $w$. Clearly
$\deg_{H'}(u)\leq \binom{|U'|-1}{2}+(|U'|-1)|W'|$.
Since $|M| \leq |W'|-1$ and $|U'|\leq n - |W'|$, it follows that
\begin{align*}
\deg_{H'}(w)+\deg_{H'}(u) &\leq \binom{2(|W'|-1)+1}{2}+\binom{n-|W'|-1}{2}+(n-|W'|-1)|W' | \\
&\leq \frac{3}{2}\left(|W'|-\frac{5}{6}\right)^2+\frac{1}{2}n^2-\frac{3}{2}n+\frac{23}{24}.
\end{align*}
Furthermore,  since $|W'| \leq  \frac{n}{\sqrt{15}}$,  we derive that
\begin{align*}
\deg_{H'}(w)+\deg_{H'}(u)  & \leq \frac{3}{2}\left( \frac{n}{\sqrt{15}}-\frac{5}{6}\right)^2+\frac{1}{2}n^2-\frac{3}{2}n+\frac{23}{24}\\
& = \frac{3}{5}n^2-\Big(\frac{3}{2}+\frac{\sqrt{15}}{6}\Big)n+2 < \Big(\frac{3}{5}+\gamma - 3\varepsilon^{13}\Big) n^2,
\end{align*}
contradicting (\ref{labelGH123212}) as $\gamma \gg \varepsilon$.
\medskip

{\noindent \bf Case 1.2} $|W'| > \frac{n}{\sqrt{15}}$.\medskip

Recall that $U\setminus V(M_1) \subseteq U'$, $|U\setminus V(M_1)| \geq 2n/3+\tau n-3 \varepsilon^{13}n$ and $|W'| \leq n/3-\tau n$. Since $|M| < |W'|$,   we obtain that $|U\setminus (V(M_1) \cup V(M))| \geq 3\tau n-3 \varepsilon^{13}n > 2$. We claim that  every vertex $u$ of $U\setminus (V(M_1) \cup V(M))$ is adjacent to some vertex of $W'$ in $H'$; Otherwise $\deg_{H'}(u) \leq \binom{|U'|-1}{2}\leq \binom{n-|W'|-1}{2}$. Since $|W'| > n/\sqrt{15}$ and $n$ is sufficiently large,
\begin{align*}
\deg_{H'}(u) \leq \binom{n-n/\sqrt{15}-1}{2} < \Big(\frac{8}{15}-\frac{1}{15}\sqrt{15}\Big)n^2  < \Big(\frac{3}{10}+\frac{1}{2}\gamma-3\varepsilon^{13} \Big)n^2,
\end{align*}
which contradicts (\ref{4}). Combining (\ref{labelGH123212}) and the degree condition of the vertices in $W'$, we obtain that
\begin{align*}
\deg_{H'}(u) &\geq \Big(\frac{3}{5}+\gamma - 3\varepsilon^{13}\Big) n^2-\Big(\frac{4}{15}+\gamma \Big)n^2  = \frac{1}{3}n^2 -3\varepsilon^{13} n^2,
\end{align*}
for every $u\in U\setminus (V(M_1) \cup V(M))$.

If no vertex in $W'_1$ is adjacent to a vertex in $U'_1$, then let $u_1 \in e^*=\{u_1,u_2^*,w^*\} \in M$ be an adjacent vertex of $w\in W'_1$. Choose $u_2 \in  U\setminus \big(V(M_1) \cup V(M)\big)$ and let $Q=\{u_1,u_2,w\}$. If some vertex $w\in W'_1$ is adjacent to a vertex $u_1 \in U'_1$, then let  $u_2 \in  U\setminus \big(V(M_1) \cup V(M)\cup \{u_1\}\big)$, $e^*=\emptyset$ and $Q=\{u_1,u_2,w\}$.
 Then we have
\begin{align}\label{5}
\sum_{u \in Q} \text{deg}_{H'}(u)   >  \Big(\frac{3}{5}+\gamma - 3\varepsilon^{13}\Big) n^2+ \frac{1}{3}n^2 -3\varepsilon^{13} n^2= \Big(\frac{14}{15}+\gamma - 6\varepsilon^{13}\Big) n^2.
\end{align}

 Obviously, $H$ contains at most $7|M|$ edges consisting of one vertex from $Q$ and two vertices from the same edge of $M$, and at most $2 \binom{|U'\setminus V(M)|}{2}$ edges with one vertex from $Q$ and two other vertices  from $U'\setminus V(M)$.

\begin{subclaim}\label{claim180}
 For any two distinct edges $e_1$, $e_2$ from $M\setminus \{e^*\}$, we have $\sum _{u \in Q} |L_{u}(e_1,e_2)| \leq 16$.
\end{subclaim}
\begin{proof} Let $H_1$ be the 3-partite subgraph of $H$ induced on three parts $Q,e_1$ and $e_2$. We observe that $H_1$ does not contain a perfect matching. Otherwise, if $u_1 \in U'_1$, then we obtain a larger matching than $M$, a contradiction to (i); if  $u_1 \notin U'_1$, then we obtain a matching $M'$ the same size to $M$ but $w^* \in W'\setminus V(M')$ is adjacent to $u_2^* \in U' \setminus V(M')$,  a contradiction  to (ii). Applying Lemma \ref{Yilu} with $k=3$, we obtain that $|E(H_1)| \leq 16$. Therefore  $\sum _{u \in Q} |L_{u}(e_1,e_2)| \leq 16$.
\end{proof}

\begin{subclaim}\label{claim190}
For any $e\in M\setminus \{e^*\}$,   we have
$\sum _{u\in Q} |L_{u}(e,V(H')\setminus V(M))| \leq \max\{6|U' \setminus V(M)|+2,5|U' \setminus V(M)|+2|W' \setminus V(M)|+2\}.$
\end{subclaim}
\begin{proof} Let $e=\{w', u_2',u_3'\} \in M\setminus \{e^*\}$ with $w' \in W'$. Apply Lemma \ref{Lemma4} with $A=U' \setminus V(M)$, $B=W' \setminus V(M)$, $G_1 = L_{w'}(V(H') \setminus V(M))$ and $G_i = L_{u_i'}(V(H') \setminus V(M))$ for $i=2,3$. Recall that $|U' \setminus V(M)| \geq 2$ and $|W' \setminus V(M)| \geq 1$. Every edge of $G_1$ intersects every edge of $G_j$  containing one vertex in $B$ for $j=2,3$, and every edge of $G_2$ containing one vertex in $B$ intersects every edge of $G_3$  containing one vertex in $B$. Otherwise, if $u_1 \in U'_1$, then we obtain a larger matching than $M$, a contradiction to (i); if  $u_1 \notin U'_1$, then we obtain a matching $M'$ the same size to $M$ but $w^* \in W'\setminus V(M')$ is adjacent to $u_2^* \in U' \setminus V(M')$,  a contradiction to (ii). By Lemma \ref{Lemma4}, we obtain $\sum _{u\in Q} |L_{u}(e,V(H') \setminus V(M))| \leq \max\{6|U' \setminus V(M)|+2,5|U' \setminus V(M)|+2|W' \setminus V(M)|+2\}.$
\end{proof}

Combining these bounds together, we obtain that
\begin{align*}
\sum_{u \in Q} \text{deg}_{H'}(u)  \leq & 16\binom{|M|}{2}+7|M| +\sum_{u\in Q}|L_{u}(V(M),V(H') \setminus V(M))|+\sum_{u\in Q}|L_{u}(U' \setminus V(M))|\\
& +\sum_{u\in Q}|L_{u}(e^*, V(H'))|.
\end{align*}
Recall that $|U'| \geq 2|W'|$. Therefore  $|U'|-2|M| \geq 2(|W'|-|M|)$. We further derive that
\begin{align*}
\sum_{u \in Q} \text{deg}_{H'}(u)  & \leq  16\binom{|M|}{2}+7|M| +|M|(6(|U'|-2|M|)+2)+2 \binom{|U'|-2|M|}{2}+9n\\
& = (2|U'|+3)|M|+|U'|^2-|U'|+9n.
\end{align*}
Since $|M| \leq n-|U'|-1$, it follows that
\begin{align*}
\sum_{u \in Q} \text{deg}_{H'}(u)  & \leq   (2|U'|+3)(n-|U'|-1)+|U'|^2-|U'|=-(|U'|-n+3)^2+n^2+6n+6.
\end{align*}
Since $|W'| > \frac{n}{\sqrt{15}}$ then $|U'| < \Big(1-\frac{1}{\sqrt{15}}\Big)n$. We obtain
\begin{align*}
\sum_{u \in Q} \text{deg}_{H'}(u)  & \leq -\Big(-\frac{n}{\sqrt{15}}+3\Big)^2+n^2+6n+6   < \frac{14}{15}n^2+8n  < \Big(\frac{14}{15}+\gamma - 6\varepsilon^{13}\Big) n^2,
\end{align*}
which contradicts (\ref{5}) as $\gamma \gg \varepsilon$.\medskip

{\noindent \medskip \bf Case 2.} Some vertex of $W'_1$ is adjacent to at least two vertices of $U'_1$.

Let $w\in W'_1$ be adjacent to $u_1,u_2 \in U'_1$, where $u_1 \neq u_2$. Clearly $\deg(w)+\deg(u_1) > \Big(\frac{3}{5}+\gamma - 3\varepsilon^{13}\Big) n^2$, $\deg(w)+\deg(u_2) > \Big(\frac{3}{5}+\gamma - 3\varepsilon^{13}\Big) n^2$. It follows that  $2\deg(w)+\deg(u_1)+\deg(u_2) > 2 \Big(\frac{3}{5}+\gamma - 3\varepsilon^{13}\Big) n^2$.

Obviously, $H$ contains at most $|M|$ edges consisting of $w$ and two vertices from the same edge of $M$, no edge consisting of $w$ and two vertices from $U'\setminus V(M)$,  at most $6|M|$ edges consisting of one vertex $\{u_1,u_2\}$ and two vertices from the same edge of $M$, and at most $ 2\binom{|U'\setminus V(M)|}{2}$ edges with one vertex from $\{u_1, u_2\}$ and two other vertices  from $U'\setminus V(M)$.

\begin{subclaim}\label{claim180}
For any two distinct edges $e_1$, $e_2$ from $M$, we have $2|L_{w}(e_1,e_2)|+|L_{u_1}(e_1,e_2)|+|L_{u_2}(e_1,e_2)| \leq 20$.
\end{subclaim}
\begin{proof} Let $H_1$ be the 3-partite subgraph of $H$ induced on three parts $\{u_1,u_2,w\},e_1$ and $e_2$. We observe that $H$ does not contain a perfect matching. Otherwise, let $M^*$ be a perfect matching of $H_1$. We know that each part of $\{u_1,u_2,w\}$, $e_1,$ $e_2$ contains exactly one vertex from $W'$,  therefore $\big(M\setminus\{e_1,e_2\}\big) \cup M^*$ is a matching larger than $M$, a contradiction to (i). Applying Lemma \ref{lemma120645778}, we obtain that  $2|L_{w}(e_1,e_2)|+|L_{u_1}(e_1,e_2)|+|L_{u_2}(e_1,e_2)| \leq 20$.
\end{proof}
\begin{subclaim}\label{claim190}
For any $e\in M$,   we have
$|L_{u_1}(e,U_1'\cup W_1')|+|L_{u_2}(e,U_1'\cup W_1')|+2|L_{w}(e,U_1'\cup W_1')| \leq \max\{8|U_1'|+2,6|U_1'|+2|W_1'|+4\}.$
\end{subclaim}
\begin{proof} Let $e=\{v_1',v_2',v_3'\} \in M$ with $v_1' \in W'$. Apply Lemma \ref{Lemma4444444444} with $A=U_1'$, $B=W_1'$, and $G_i = L_{v_i'}(U_1'\cup W_1')$ for $i=1,2,3$. Recall that $|U_1'| \geq 2$ and $|W_1'| \geq 1$. Every edge of $G_1$ intersects every edge of $G_j$  containing one vertex in $B$ for $j=2,3$, and every edge of $G_2$ containing one vertex in $B$ intersects every edge of $G_3$  containing one vertex in $B$. Otherwise, letting $e_1$ and $e_2$ be two such disjoint edges.  Then $(M\setminus\{e\}) \cup \{e_1,e_2\}$ is a larger matching than $M$, a contradiction to (i).
By Lemma \ref{Lemma4444444444}, we obtain $|L_{u_1}(e,U_1'\cup W_1')|+|L_{u_2}(e,U_1'\cup W_1')|+2|L_{w}(e,U_1'\cup W_1')| \leq \max\{8|U_1'|+2,6|U_1'|+2|W_1'|+4\}.$
\end{proof}

Combining these bounds together, we obtain that
\begin{align*}
2\deg(w)+\deg(u_1)+\deg(u_2)  \leq & 20\binom{|M|}{2}+8|M|+|L_{u_1}(U' \setminus V(M))|+|L_{u_2}(U' \setminus V(M))| \\
&+|L_{u_1}(V(M),V(H') \setminus V(M))|+|L_{u_2}(V(M),V(H') \setminus V(M))|\\
&+2|L_{w}(V(M),V(H') \setminus V(M))|.
\end{align*}

Recall that $|U'| \geq 2|W'|$. Therefore  $|U'|-2|M| \geq 2(|W'|-|M|)\geq2$. We further derive that
\begin{align*}
2\deg(w)+\deg(u_1)+\deg(u_2)  & \leq  20\binom{|M|}{2}+8|M| +|M|(8(|U'|-2|M|)+2)+2 \binom{|U'|-2|M|}{2}\\
&=|U'|^2+(4|M|-1)|U'|-2|M|^2+2|M|.
\end{align*}
Note that the quadratic function $ x^2+(4|M|-1)x$ is minimized at $x= -2|M|+\frac{1}{2}$.  Since $|U'| \leq n-|W'|$, it follows that this function is maximized at $|U'|= n-|W'|$.
\begin{align*}
2\deg(w)+\deg(u_1)+\deg(u_2)  & \leq  (n-|W'|)^2+(4|M|-1)(n-|W'|)-2|M|^2+2|M|.\\
& = -2|M|^2+(4n-4|W'|+2)|M|+n^2-2n|W'|+|W'|^2-n+|W'|.
\end{align*}
Note that the quadratic function $  -2x^2+(4n-4|W'|+2)x$ is maximized at $x= n-|W'|+\frac{1}{2}$. Since $|M| \leq |W'|-1 \leq n-|W'|+\frac{1}{2}$, it follows that
\begin{align*}
  2\deg(w)+\deg(u_1)+\deg(u_2) \leq   & -2(|W'|-1)^2+(4n-4|W'|+2)(|W'|-1)\\
  & +n^2-2n|W'|+|W'|^2-n+|W'|\\
  = & -5|W'|^2+(2n+11)|W'|+n^2-5n-4.
\end{align*}
Note that the quadratic function $  -5x^2+(2n+11)x$ is maximized at $x= \frac{1}{5}n+\frac{11}{10}$. Let $|W'|= \frac{1}{5}n+\frac{11}{10}$, then we  obtain that
\begin{align*}
  2\deg(w)+\deg(u_1)+\deg(u_2) \leq   \frac{6}{5}n^2-\frac{14}{5}n+\frac{41}{20} < 2 \Big(\frac{3}{5}+\gamma - 3\varepsilon^{13}\Big) n^2.
\end{align*}
This is a contradiction. This finishes the proof of Claim \ref{Claimcoveringalllsmaler}. \end{proof}

We continue to prove Lemma \ref{lemma11182}. Let $M_2$ be a matching satisfying the following three conditions: (1) $M_2$ covers all the vertices of $A_1$; (2) Subject to (1), the matching $M_2$ is as large as possible; (3) Subject to (1) and (2), the size of $A_3 \cup A_4$ is as large as possible.  By Claim \ref{Claimcoveringalllsmaler}, such $M_2$ indeed exists. Let $M_{21}=\{e\in M_2: e\cap A_1 \neq \emptyset\}$ and  $M_{22} = M_2 \setminus M_{21}$.
It suffices to show that $|V(H') \setminus V(M_2)| \leq C$.  Indeed, if this is true, the absorbing matching $M_1$ can absorb all the vertices of $V(H') \setminus V(M_2)$ to form a matching $M_3$, then $M_2 \cup M_3$ forms a perfect matching in $H$. We have the following claims:
\begin{claim}
$|A_2\setminus V(M_2)| \leq C/3.$
\end{claim}
\begin{proof}
Let $M_{22}'=\{e \in M_{22}: e \cap A_2 \neq \emptyset \}$ and $M_{22}''= M_{22}\setminus M_{22}'$. Clearly  a vertex in $A_2$ is not adjacent to any vertex in $A_1 \cup A_2$.   Let $u\in A_2\setminus V(M_2)$. For any $EF \in \binom{M_{21}\cup M_{22}'}{2}$, we have $|L_{u}(EF)| \leq 4$; For any $E \in M_{21},F \in M_{22}''$, we have $|L_{u}(EF)| \leq 6$.
Let $$A_2' =\Big\{u \in A_2\setminus V(M_2):  \Big|\Big\{\{E_1E_2\}: \big|L_u(E_1E_2)\big| \geq 4, E_1 \in M_{22}',E_2 \in M_{22}'' \Big\}\Big| \geq \rho n^2  \Big\}.$$
We claim that $|A_2'| \leq C/9.$ Otherwise let $G$ be the bipartite graph with vertex classes $A_2'$ and $\Big\{\{E_1E_2\}: E_1 \in M_{22}',E_2 \in M_{22}''\Big\}$ where $\{u,\{E_1E_2\}\}$ is an edge if and only if $\Big|L_u(E_1E_2)\Big| \geq 4$. Clearly $G$ contains at least $|A_2'|\rho n^2$ edges. Then $G$ contains a vertex $E_1E_2$ such that $d_G(E_1E_2) \geq C \rho/9 >2$. Suppose $u_1,u_2\in A_2'$ such that $|L_{u_1}(E_1E_2)| \geq 4$, $|L_{u_2}(E_1E_2)| \geq 4$. Clearly $L_{u_1}(E_1E_2)$ and $L_{u_2}(E_1E_2)$ have a rainbow 2-matching.  Let $E_1=\{v_1,v_2,v_3\}$ and $E_2=\{v_4,v_5,v_6\}$.  Suppose that $E_1'=\{u_1, v_1,v_4\}$ and $E_2'=\{u_2, v_2,v_5\}$ are a rainbow 2-matching. Then $(M_2\setminus \{E_1,E_2\}) \cup \{E_1',E_2'\}$ is a matching of the same size to $M_2$ and the two vertices $u_1,u_2 \in A_2'$ are replaced by $v_3,v_6$, where $v_6 \in A_3 \cup A_4$, a contradiction to (3).
Let $$A_2'' =\Big\{u \in A_2\setminus V(M_2):  \Big|\Big\{\{E_1E_2\}: |L_u(E_1E_2)| \geq 4, E_1,E_2  \in M_{22}'' \Big\}\Big| \geq \rho n^2  \Big\}.$$
We can obtain $|A_2''| \leq C/9$ similarly. Let $A_2'''=(A_2 \setminus V(M_2))\setminus (A_2' \cup A_2'').$ We prove that $|A_2'''|  \leq C/9$. To the contrary, let $Q=\{u_1,u_2,u_3\} \subseteq A_2'''$. Then $\sum_{u \in Q} \text{deg}_{H'}(u)  > \big(4/5+3\gamma \big)n^2$.

By the same argument of Claim \ref{claim18},  we can obtain the following claim.
\begin{subclaim}\label{claim188888}
 For any $e \in M_{21}$, we have $\sum_{i=1}^3|L_{u_i}(e,V(H')\setminus V(M_2) )| \leq 6|V(H')\setminus V(M_2)|$.
\end{subclaim}

By the same argument of Claim \ref{claim19}, we can obtain the following claim.
\begin{subclaim}\label{claim1999999} For any $e \in M_{22}$, we have $ \sum_{i=1}^3|L_{u_i}(e,V(H')\setminus V(M_2) )| \leq 3(|V(H')\setminus V(M_2)|+3)$.
\end{subclaim}
Note that Subclaims \ref{claim188888} and \ref{claim1999999} also hold when $u_1,u_2,u_3 \in A_3 \cup A_4$.
Combining these bounds together, we obtain that
\begin{align*}
\sum_{i=1}^{3} \text{deg}_{H'}(u_i)  \leq & 12\binom{|M_{21}\cup M_{22}'|}{2}+9\binom{|M_{22}''|}{2}
+18\binom{|M_{21}|}{1}\binom{|M_{22}''|}{1}+9\binom{|M_{22}'|}{1}\binom{|M_{22}''|}{1}+81\rho n^2 \\
&+9|M_{2}|+\sum_{i=1}^3|L_{u_i}(V(M_{21}\cup M_{22}),V(H')\setminus V(M_2))|\\
\leq & 3\binom{|M_{21}\cup M_{22}'|}{2}+9\binom{|M_{2}|}{2}
+9\binom{|M_{21}|}{1}\binom{|M_{22}''|}{1}+81\rho n^2 \\
&+9|M_{2}|+6|M_{21}|(n-3|M_2|)+3|M_{22}|(n-3|M_2|+3).
\end{align*}
Since $|M_{22}''|=|M_2|-|M_{21}|-|M_{22}'|$ and $|M_{22}|=|M_2|-|M_{21}|$,

\begin{align*}
\sum_{i=1}^{3} \text{deg}_{H'}(u_i)  \leq -\frac{9}{2}|M_2|^2+\Big(\frac{27}{2}+3n\Big)|M_2|+3|M_{21}|n-6|M_{21}||M_{22}'|-\frac{15}{2}|M_{21}|^2
+\frac{3}{2}|M_{22}'|^2+81\rho n^2.
\end{align*}
Note that the quadratic function $-\frac{9}{2}x^2+\Big(\frac{27}{2}+3n\Big)x$ is maximized at $x= \frac{1}{3}n+\frac{3}{2}$. Since $|M_2| \leq  \frac{1}{3}n-1$, it follows that
\begin{align*}
\sum_{i=1}^{3} \text{deg}_{H'}(u_i)  \leq \frac{3}{2}|M_{22}'|^2-6|M_{21}||M_{22}'|+3|M_{21}|n-\frac{15}{2}|M_{21}|^2+\frac{9}{2}n+\frac{1}{2}n^2+81\rho n^2\stackrel{\triangle}{=}f(M_{22}').
\end{align*}
Note that the quadratic function $\frac{3}{2}x^2-6|M_{21}|x$ is minimized at $x= 2|M_{21}|$. Since $0 \leq |M_{22}'| \leq \frac{1}{3}n-|M_{21}|$, we obtain $\sum_{i=1}^{3} \text{deg}_{H'}(u_i)  \leq \max\{f(0),f(\frac{1}{3}n-|M_{21}|)\}$. However,
\begin{align*}
f(0)  & = -\frac{15}{2}\Big( |M_{21}|-\frac{1}{5}n\Big)^2+\frac{4}{5}n^2+\frac{9}{2}n+81\rho n^2 \leq \frac{4}{5}n^2+\frac{9}{2}n+81\rho n^2 < \Big(\frac{4}{5}+3\gamma \Big)n^2
\end{align*}
and
\begin{align*}
f\Big(\frac{1}{3}n-|M_{21}|\Big)  & = \frac{2}{3}n^2+\frac{9}{2}n+81\rho n^2  < \Big(\frac{4}{5}+3\gamma \Big)n^2.
\end{align*}
It is a contradiction. Therefore $|A_{2}'''| \leq C/9$ and $|A_2\setminus V(M_2)| \leq C/3$.
\end{proof}

\begin{claim}
$|A_3\setminus V(M_2)| \leq C/3.$
\end{claim}
\begin{proof} Clearly a vertex in $A_3$ is not adjacent to any vertex in $A_1$.   Let $u\in A_3\setminus V(M_2)$.  For any $EF \in \binom{M_{21}}{2}$, we have $|L_{u}(EF)| \leq 4$; For any $E \in M_{21},F \in M_{22}$, we have $|L_{u}(EF)| \leq 6$.
Let
$$A_{31} =\Big\{u \in A_3\setminus V(M_2):  \Big|\Big\{\{E_1E_2\}: E_1,E_2 \in M_{22}, L_u(E_1E_2) \,\, \text{ has a perfect matching}   \Big\}\Big| \geq \frac{1}{2}\rho n^2  \Big\}.$$
By Lemma \ref{perfectmatching}, we have $|A_{31}| \leq C/9.$ Let
$$A_{32} =\Big\{u \in A_3\setminus V(M_2):  \Big|\Big\{\{E_1E_2\}: E_1,E_2 \in M_{22}, L_u(E_1E_2) \cong B_{033} \Big\}\Big| \geq \frac{1}{2}\rho n^2  \Big\}.$$
By Lemma \ref{b023b033}, we have $|A_{32}| \leq C/9.$  By Fact \ref{fact1}, if $|L_u(E_1E_2)| \geq  6$, then $L_u(E_1E_2)$ has a perfect matching or $L_u(E_1E_2)\cong B_{033}$.  Let $A_{33} =(A_3\setminus V(M_2))\setminus (A_{31} \cup A_{32})$. For any $u \in A_{33}$, we obtain that  $M_{22}$ contains at most $\rho n^2$ pairs $E_1E_2$ satisfying $|L_{u}(E_1E_2)|\geq 6 $. Otherwise $u \in A_{31}$ or $u \in A_{32}$, a contradiction.   We claim that $|A_{33}| \leq C/9.$  To the contrary, suppose that $|A_{33}| > C/9.$  Let $u_1,u_2,u_3 \in A_{33}$. We have
\begin{align*}
\sum_{i=1}^{3} \text{deg}_{H'}(u_i)  >  \Big(\frac{9}{10}+\frac{3}{2}\gamma- \frac{9}{2}\varepsilon^{13}\Big)n^2 .
\end{align*}

By  Subclaims \ref{claim188888} and \ref{claim1999999}, we obtain that
\begin{align*}
\sum_{i=1}^{3} \text{deg}_{H'}(u_i)  \leq &   12\binom{|M_{21}|}{2}+18\binom{|M_{21}|}{1}\binom{|M_{22}|}{1}+15\binom{|M_{22}|}{2}+27\rho n^2+9 (|M_{21}|+|M_{22}|) \\
&+6|M_{21}|(n-3(|M_{21}|+|M_{22}|))+3|M_{22}|(n-3(|M_{21}|+|M_{22}|)+3)\\
=& -12|M_{21}|^2+(3-9|M_{22}|+6n)|M_{21}|-\frac{3}{2}|M_{22}|^2+\frac{21}{2}|M_{22}|+3|M_{22}|n+27\rho n^2.
\end{align*}
Note that the quadratic function $-12x^2+(3-9|M_{22}|+6n)x$ is maximized at $x= \frac{1}{4}n-\frac{3}{8}|M_{22}|+\frac{1}{8}$.

If $|M_{22}| \geq \frac{2}{15}n$, then $\frac{1}{4}n-\frac{3}{8}|M_{22}|+\frac{1}{8} \geq \frac{1}{3}n-|M_{22}|$. Since $|M_{21}| \leq \frac{1}{3}n-|M_{22}|$, it follows that
\begin{align*}
\sum_{i=1}^{3} \text{deg}_{H'}(u_i)  \leq & -12\Big(\frac{1}{3}n-|M_{22}|\Big)^2+(3-9|M_{22}|+6n)\Big(\frac{1}{3}n-|M_{22}|\Big)\\
& -\frac{3}{2}|M_{22}|^2+\frac{21}{2}|M_{22}|+3|M_{22}|n+27\rho n^2. \\
 = & -\frac{9}{2}|M_{22}|^2+\Big(2n+\frac{15}{2}\Big)|M_{22}|+\frac{2}{3}n^2+n+27\rho n^2.
\end{align*}
Note that the quadratic function $-\frac{9}{2}x^2+\Big(2n+\frac{15}{2}\Big)x$ is maximized at $x= \frac{2}{9}n+\frac{5}{6}$. Therefore
\begin{align*}
\sum_{i=1}^{3} \text{deg}_{H'}(u_i)
& \leq -\frac{9}{2}\Big(\frac{2}{9}n+\frac{5}{6}\Big)^2+\Big(2n+\frac{15}{2}\Big)\Big(\frac{2}{9}n+\frac{5}{6}\Big)+\frac{2}{3}n^2+n+27\rho n^2 \\
& \leq \frac{8}{9}n^2+\frac{8}{3}n+\frac{25}{8}+27\rho n^2 < \Big(\frac{9}{10}+\frac{3}{2}\gamma- \frac{9}{2}\varepsilon^{13}\Big)n^2 .
\end{align*}
It is a contradiction. So $|M_{22}| < \frac{2}{15}n$, then
\begin{align*}
\sum_{i=1}^{3} \text{deg}_{H'}(u_i)   \leq &   -12\Big(\frac{1}{4}n-\frac{3}{8}|M_{22}|+\frac{1}{8}\Big)^2+(3-9|M_{22}|+6n)\Big(\frac{1}{4}n-\frac{3}{8}|M_{22}|
+\frac{1}{8}\Big)\\
 & -\frac{3}{2}|M_{22}|^2+\frac{21}{2}|M_{22}|+3|M_{22}|n+27\rho n^2\\
 = & \frac{3}{16}|M_{22}|^2+\Big(\frac{3}{4}n+\frac{75}{8}\Big)|M_{22}|+\frac{3}{4}n^2+\frac{3}{4}n+\frac{3}{16}+27\rho n^2.
\end{align*}
Clearly
\begin{align*}
\sum_{i=1}^{3} \text{deg}_{H'}(u_i)   \leq &  \frac{3}{16}\Big(\frac{2}{15}n\Big)^2+\Big(\frac{3}{4}n+\frac{75}{8}\Big)\Big(\frac{2}{15}n\Big)+\frac{3}{4}n^2+\frac{3}{4}n+\frac{3}{16}+27\rho n^2\\
= &  \frac{64}{75}n^2+2n+\frac{3}{16}+27\rho n^2  <  \Big(\frac{9}{10}+\frac{3}{2}\gamma- \frac{9}{2}\varepsilon^{13}\Big)n^2 .
\end{align*}
It is a contradiction. Therefore $|A_{33}| \leq C/9$ and $|A_3\setminus V(M_2)| \leq C/3$.
\end{proof}

\begin{claim}
$|A_4\setminus V(M_2)| \leq C/3.$
\end{claim}
\begin{proof}
Let
$$A_{41} =\Big\{u \in A_4\setminus V(M_2):  \Big|\Big\{\{E_1E_2\}: E_1,E_2 \in M_{21}\cup M_{22}, L_u(E_1E_2) \,\, \text{ has a perfect matching}   \Big\}\Big| \geq \frac{1}{2}\rho n^2  \Big\}.$$
By Lemma \ref{perfectmatching}, we have $|A_{41}| \leq C/9.$  We further let
$$A_{42} =\Big\{u \in A_4\setminus V(M_2):  \Big|\Big\{\{E_1E_2\}: E_1,E_2 \in M_{22}, L_u(E_1E_2) \cong B_{033} \Big\}\Big| \geq \frac{1}{2}\rho n^2  \Big\}.$$
By Lemma \ref{b023b033}, we have $|A_{42}| \leq C/9.$  Let $A_{43} =(A_4\setminus V(M_2))\setminus (A_{41} \cup A_{42})$.   By Fact \ref{fact1}, if $|L_u(E_1E_2)| \geq  6$, then $L_u(E_1E_2)$ has a perfect matching or $L_u(E_1E_2)\cong B_{033}$. Therefore for any vertex $u \in A_{43}$,  $$ \Big|\Big\{\{E_1E_2\}: E_1 \in M_{21},E_2 \in M_{21} \cup M_{22}, |L_u(E_1E_2)| \geq 7 \Big\}\Big| \leq \rho n^2,$$ and
$$ \Big|\Big\{\{E_1E_2\}: E_1,E_2 \in M_{22}, |L_u(E_1E_2)| \geq 6 \Big\}\Big| \leq \rho n^2.$$
We claim that $|A_{43}| \leq C/9.$ To the contrary, suppose that $|A_{43}| > C/9.$ Let $u_1,u_2,u_3 \in A_{43}$, we have
\begin{align*}
\sum_{i=1}^{3} \text{deg}_{H'}(u_i)  > n^2-9 \varepsilon^{13} n^2.
\end{align*}
By  Subclaims \ref{claim188888} and \ref{claim1999999}, we obtain
\begin{align*}
\sum_{i=1}^{3} \text{deg}_{H'}(u_i)  \leq & 18\binom{|M_{21}|}{2}+18\binom{|M_{21}|}{1}\binom{|M_{22}|}{1}+27\rho n^2+15\binom{|M_{22}|}{2}+27\rho n^2+9 (|M_{21}|+|M_{22}|) \\
&+6|M_{21}|(n-3(|M_{21}|+|M_{22}|))+3|M_{22}|(n-3(|M_{21}|+|M_{22}|)+3)\\
=&  -\frac{3}{2}|M_{22}|^2+\Big(-9|M_{21}|+\frac{21}{2}+3n\Big)|M_{22}|+6|M_{21}|n-9|M_{21}|^2+54\rho n^2.
\end{align*}
Note that the quadratic function $-\frac{3}{2}x^2+(-9|M_{21}|+\frac{21}{2}+3n)x$ is maximized at $x= n-3|M_{21}|+\frac{7}{2}$.  Since $|M_{22}| \leq \frac{1}{3}n-|M_{21}| \leq n-3|M_{21}|+\frac{7}{2}$,  it follows that
\begin{align*}
\sum_{i=1}^{3} \text{deg}_{H'}(u_i)  \leq & -\frac{3}{2}\Big(\frac{1}{3}n-|M_{21}|\Big)^2+\Big(-9|M_{21}|+\frac{21}{2}+3n\Big)\Big(\frac{1}{3}n-|M_{21}|\Big)+6|M_{21}|n-9|M_{21}|^2+54\rho n^2\\
= & -\frac{3}{2}|M_{21}|^2+\Big(n-\frac{21}{2}\Big)|M_{21}|+\frac{5}{6}n^2+\frac{7}{2}n+54\rho n^2.
\end{align*}
Note that the quadratic function $ -\frac{3}{2}x^2+\Big(n-\frac{21}{2}\Big)x$ is maximized at $x= \frac{1}{3}n-\frac{7}{2}$.
Moreover, for any vertex $u \in U$, $\deg_{H'}(u) \geq \deg_{H}(u) - |V(M_1)|n > \frac{3}{10}n^2+\frac{1}{2}\gamma n^2 - 3\varepsilon^{13} n^2$. Therefore  $U \cap A_1 =\emptyset$. It follows that $|M_{21}|=|A_1| \leq |W| \leq n/3-\tau n$. We obtain that
\begin{align*}
\sum_{i=1}^{3} \text{deg}_{H'}(u_i)  \leq & -\frac{3}{2}\Big(\frac{n}{3}-\tau n \Big)^2+\Big(n-\frac{21}{2}\Big)\Big(\frac{n}{3}-\tau n \Big)+\frac{5}{6}n^2+\frac{7}{2}n+54\rho n^2\\
& = \Big(1-\frac{3}{2}\tau^2\Big)n^2+\frac{21}{2}\tau n+54\rho n^2  < n^2-9 \varepsilon^{13} n^2.
\end{align*}
It is a contradiction. Therefore $|A_{43}| \leq C/9$ and $|A_4\setminus V(M_2)| \leq C/3$. It follows that $|V(H') \setminus V(M_2)| \leq C$. This finishes the proof.
\end{proof}

\end{document}